\documentclass[12pt,leqno]{amsart}
\usepackage{amsfonts,amsthm,amsmath,comment,xcolor,siunitx,multirow,amssymb}
\theoremstyle{plain}
\newtheorem{thm}{Theorem}[section]

\newtheorem{lem}[thm]{Lemma}

\theoremstyle{definition}
\newtheorem{df}[thm]{Definition}
\newtheorem{rem}[thm]{Remark}
\newtheorem{ex}[thm]{Example}

\usepackage{comment}

\newcommand{\FF}{\mathbb{F}}

\newcommand{\CC}{\mathbb{C}}
\newcommand{\allone}{\mathbf{1}}
\newcommand{\0}{\mathbf{0}}

\newcommand{\II}{\mathrm{II}}
\newcommand{\III}{\mathrm{III}}
\newcommand{\IV}{\mathrm{IV}}

\def\bm#1{\mathbf{#1}}

\DeclareMathOperator{\wt}{wt}

\begin{document}

\title[Neighbors, neighbor graphs and invariant rings]{{
Neighbors, neighbor graphs  and invariant rings in coding theory 
%\footnote{This work was supported by JSPS KAKENHI (18K03217, 17K05164).}
%from the point of view of Oura's conjecture 
}
%\footnote{
%This work was supported by JSPS KAKENHI Grant Number 22840003, 24740031.
%}
%\\
%{\bf Version 1.4} 
}

\author[Chakraborty]{Himadri Shekhar Chakraborty*}
\thanks{*Corresponding author}
\address
	{
		Department of Mathematics, 
		Shahjalal University of Science and Technology, Sylhet-3114, Bangladesh\\
	}
\email{himadri-mat@sust.edu}

\author[Chiari]{Williams Chiari}
\address
	{
		Graduate School of Natural Science and Technology,
		Kanazawa University, Ishikawa 920-1192, Japan\\
	}
\email{williamschiari@gmail.com}

%\author[Ishikawa]{Reina Ishikawa}
%\address
%{
%	Graduate School of Science and Engineering, 
%	Waseda University, 
%	Tokyo 169-8555, Japan\\ 
%}
%\email{reina.i@suou.waseda.jp}

\author[Miezaki]{Tsuyoshi Miezaki}
\address
	{
		Faculty of Science and Engineering, 
		Waseda University, 
		Tokyo 169-8555, Japan\\
	}
\email{miezaki@waseda.jp} 

\author[Oura]{Manabu Oura}
\address
	{
			Faculty of Mathematics and Physics, 
			Kanazawa University,  
			Ishikawa 920-1192, Japan
	}
\email{oura@se.kanazawa-u.ac.jp} 

\date{}
\maketitle

\begin{abstract}
In the present paper, 
we discuss the class of Type~$\III$ and Type~$\IV$ codes 
from the perspectives of neighbors. 
Our investigation analogously extends 
the results originally presented by 
Dougherty~\cite{Dougherty2022}
concerning the neighbor graph of binary self-dual codes.
Moreover, as an application of neighbors in invariant theory,
we show that the ring of 
the weight enumerators of Type~$\II$ 
code $d_{n}^{+}$ and its neighbors
in arbitrary genus is finitely generated.
Finally, we obtain a minimal set of generators of this ring
up to the space of degree~$24$ and genus~$3$.
\end{abstract}

{\small
\vspace{-0.15in}

\quad
}

%\noindent
%2010 {\it Mathematics Subject Classification}.
%Primary 11F30;
%Secondary 20D08, 11F27.\\ \quad

\section{Introduction}

One of the most celebrated classifications of codes 
in algebraic coding theory
is self-dual codes.
The study of this type of codes is immensely significant
not only because of its various practical importance,
as many of the best-known codes are self-dual,
but also its diverse 
theoretical connections with geometric lattices, block designs
and invariant theory. 
For instance, see~\cite{AsMa69, CS1999, NRS}.
Brualdi and Pless~\cite{BP1991}
introduced the concept of neighbors, 
a remarkable notion in the theory of
binary self-dual codes. 
Two binary self-dual codes of length~$n$
is known as neighbors if they share a subcode of 
codimension~$1$. 
In a recent study, 
Dougherty~\cite{Dougherty2022}
defined neighbor graph of binary self-dual codes,
where two codes are connected by an edge if and only if
they are neighbors.

The main purpose of this paper is to extend the 
results in~\cite{Dougherty2022}
to the case of non-binary self-dual codes, 
namely for Type~$\III$ and Type~$\IV$ codes.
We define the notion
of neighbors for the self-dual codes over
any finite field as follows:

\begin{df}
	Two self-dual codes of length~$n$ over~$\FF_{q}$
	are called \emph{neighbors} if they share a subcode
	of dimension $\frac{n}{2}-1$,
	that is,
	their intersection is a subcode of codimension~$1$.
\end{df}

In this note, 
we discuss some properties of neighbors in different classes
of self-dual codes, specifically, Type~$\III$ codes, 
Type~$\III$ codes containing all-ones vector and Type~$\IV$ codes.
We refer the readers to~\cite{CPS1979, MOSW1978, PP1973} 
for a detail discuss on these types of codes.
%Type~$\III$ and Type~$\IV$ codes.
We also define the neighbor graphs of above mentioned classes
of self-dual codes.
We show some significant applications of
the neighbor graphs 
in the study of Type~$\III$ and Type~$\IV$ codes
and their neighbors.
In particular, 
we use these graphs to count the number of 
Type~$\III$ (resp. Type~$\IV$) codes
applying the concept of~$k$-neighbors.
Moreover, we use the idea of $k$-neighbors
of self-dual codes over finite fields 
to define the notion of $k$-neighbor graphs.
Using this notion we derive several analogous
results of counting formulae.

\begin{df}
	For $0 \leq k \leq \frac{n}{2}$, 
	two self-dual codes of length~$n$ over~$\FF_{q}$
	are called $k$-neighbors if and only if they share a subcode
	of dimension~$\frac{n}{2}-k$.
\end{df}

Finally, we apply neighbors in invariant theory
and prove that the ring of 
the weight enumerators of Type~$\II$ 
code $d_{n}^{+}$ and its neighbors
in arbitrary genus can be finitely generated over~$\CC$.
We also show that the space of degree~$24$ of this ring is 
strictly smaller than the ring of the weight enumerators of
all Type~$\II$ codes.

This paper is organized as follows. 
In Section~\ref{Sec:Preli}, 
we discuss definitions and the basic properties of linear codes, 
neighbors and graphs that are needed to understand this paper. 
In Section~\ref{Sec:TypeIIICodes}, 
we define the neighbor graph of Type~$\III$ codes,
and answer various counting questions in this graph. 
Using $k$-neighbors we also derive a new formula 
to count the number of Type~$\III$ codes.
We illustrate these results in Sections~\ref{Sec:TypeIIICodesAllOne} 
and~\ref{Sec:TypeIVCodes} 
for other important classes of self-dual codes,
for example, Type~$\III$ codes containing all-ones vector
and Type~$\IV$ codes. 
In Section~\ref{Sec:kNeighborGraphs},
we discuss $k$-neighbor graphs and its properties.
Finally, in Section~\ref{Sec:Application},
we discuss the invariant ring of the weight
enumerators of Type~$\II$ code~$d_{n}^{+}$
and its neighbors.

All computer calculations in this paper were done with the help of Magma~\cite{Magma}
and SageMath~\cite{SageMath}.
%and Mathematica~\cite{Mathematica}.

\section{Preliminaries}\label{Sec:Preli}

In this section,
we give a brief discussion on linear codes and graphs 
including the basic definitions and properties.
We follow~\cite{HP2003, MS1977} 
for the discussions.

\subsection{Linear codes}

Let $\FF_{q}$ be a finite field of order~$q$,
where $q$ is a prime power. 
In this paper, $q$ will be either~$2$, $3$ or~$4$.
Then $\FF_{q}^{n}$ denotes the vector space of dimension~$n$ 
with inner product:
\[
	u\cdot v 
	:= 
	\begin{cases}
		u_{1}v_{1} + \cdots + u_{n}v_{n},
		& 
		\mbox{if $q =2, 3$}\\
		u_{1}v_{1}^{2} + \cdots + u_{n}v_{n}^{2},
		&
		\mbox{if $q = 4$}
	\end{cases}
\]
for $u,v \in \FF_{q}^n$,
where
$u = (u_{1},\ldots,u_{n})$ and $v = (v_{1},\ldots,v_{n})$.
Here $\FF_{4} := \{0,1,\omega,\omega^2\}$
with $1+\omega+\omega^2 = 0$. 
We call $u$ and $v$ are \emph{orthogonal}
If $u \cdot v = 0$. 
An element
$u \in \FF_{q}^n$ is called \emph{self-orthogonal} if 
$u \cdot u = 0$.
We denote the all-ones vector by~$\allone$
and zero vector by~$\0$.
The \emph{weight} $\wt(u)$ of a vector $u \in \FF_{q}^{n}$
is the number of non-zero coordinates in it.
An \emph{$\FF_q$-linear code}~$C$ of length~$n$ 
is a vector subspace of $\FF_{q}^{n}$.  
The elements of~$C$ are called \emph{codewords}.
The \emph{dual code} of~$C$ is defined as
\[
	C^\perp 
	:= 
	\{
		v\in \FF_{q}^{n} 
		\mid 
		u \cdot v = 0 
		\text{ for all } 
		u\in C
\}. 
\]
If $C \subseteq C^\perp$, then $C$ is called \emph{self-orthogonal}.
Clearly, every codeword of a self-orthogonal code is self-orthogonal. 
In addition, when $C = C^\perp$, we call $C$ \emph{self-dual}. 
It is well known that 
the length~$n$ of a self-dual code over~$\FF_q$ is even 
and the dimension is $n/2$. 

\begin{lem}\label{Lem:SlfOrthoTypeIII}
	Let $n \equiv 0\pmod 4$. 
	Then the weight of any self-orthogonal vector 
	in~$\FF_{3}^{n}$ is divisible by~$3$.
\end{lem}

\begin{proof}
	Since $n \equiv 0\pmod 4$ 
	and $x^2 = 1$ for any non-zero $x \in \FF_{3}$,
	therefore any vector in~$\FF_{3}^{n}$
	is self-orthogonal if and only if its weight is divisible by~$3$.
\end{proof}

\begin{lem}\label{Lem:SlfOrthoTypeIV}
	Let $n$ is even. Then the weight of each self-orthogonal vector 
	in~$\FF_{4}^{n}$ is even.
\end{lem}

\begin{proof}
	Since $n$ is even and $x^{3} = 1$ for any non-zero
	$x \in \FF_{4}$,
	therefore any vector in~$\FF_{4}^{n}$
	is self-orthogonal if and only if its weight is even.
\end{proof}

For any self-dual code~$C$ of length~$n$ over~$\FF_{q}$,
it is immediate that 
$$C_{0} := \{{w}\in C \mid {w} \cdot {v} = 0 \}$$
is a subcode of~$C$ with co-dimension~$1$,
where ${v} \in \FF_{q}^{n}$ is a self-orthogonal 
vector not in~$C$. Then it is not hard to show that 
%for any self-dual code~$C$ of length~$n$ over~$\FF_{q}$ and a self-orthogonal 
%vector~$v \in \FF_{q}^{n}$ not in~$C$,
$
	N_{C}({v})
	:=
	\langle
%	\{
%	{w}\in C
%	\mid
%	{w} \cdot {v}
%	=
%	0
%	\}, 
	C_{0},{v}
	\rangle
$
is a neighbor of~$C$, see~\cite{Dougherty2022}.

In general, a self-dual code~$C$ of length~$n$ over~$\FF_{q}$
has several neighbors.
We do not always have that 
$N_{C}(v_{1}) \neq N_{C}(v_{2})$ 
for $v_{1} \neq v_{2}$.
Then by the similar arguments in~\cite[Lemma 3.2]{Dougherty2022},
we have the following useful lemma.

\begin{lem}\label{Lem:EqualNeighbors}
	Let $C$ be a self-dual code of length~$n$ over~$\FF_{q}$.
	Let ${v_{1}}$ and $v_{2}$ be self-orthogonal vectors
	in~$\FF_{q}^{n}$ but not in~$C$.
	Then $N_{C}({v_{1}}) = N_{C}(v_{2})$
	if and only if there exists a vector~$w \in C$
	such that $w \cdot v_{1} = 0$ 
	and $v_{2} = w + \alpha v_{1}$
	for any nonzero $\alpha \in \FF_{q}$.
\end{lem}

\begin{lem}\label{Lem:NumEqualNeighbor}
	Let~$C$ be a self-dual code of length~$n$ over~$\FF_{q}$.
	Let~$v_{0} \in \FF_{q}^{n}$ be a self-orthogonal vector
	not in~$C$.
	Then the number of self-orthogonal vectors~$v \in \FF_{q}^{n}$
	such that $N_{C}(v) = N_{C}(v_{0})$
	is $(q-1)q^{\frac{n}{2}-1}$.
\end{lem}

\subsection{Graphs}

A \emph{graph} $G := (V,E)$ consists of~$V$,
a non-empty set of \emph{vertices}, 
and $E$, a set of \emph{edges}.
An edge is usually incident with two vertices.
But if the edge incident 
with equal end vertices, the edge is called a \emph{loop}.
A graph is called \emph{simple}
if it has neither loops nor multiple edges.
The \emph{degree} of a vertex~$v$ in graph~$G$,
denoted by~$\deg_{G}(v)$,
is the number of edges that are incident to~$v$. 
The graph~$G$ is called \emph{regular} if $\deg_{G}(v)$
is same for each vertex~$v$ in~$G$.
A \emph{path} in $G$ is a sequence of edges 
$(e_1, e_2, \ldots, e_{m - 1})$ 
having a sequence of vertices 
$(v_1, v_2,\ldots, v_n)$ 
satisfying $e_i \mapsto \{v_i, v_{i + 1}\}$ 
for $i = 1, 2,\ldots, m - 1$.
If there is a path between any two vertices of a graph,
then the graph is called \emph{connected}.

\section{Type~$\III$ codes}\label{Sec:TypeIIICodes}

A self-dual code over~$\FF_3$ 
is called \emph{Type~$\III$}
if the weight of each codeword 
is congruent to $0 \pmod 3$.
We recall that Type~$\III$ 
codes of length~$n$ exists if and only if 
$n\equiv 0\pmod 4$.
Let $T_{\III}(n)$ be the number of Type~$\III$ codes of 
length~$n \equiv 0 \pmod 4$. 
Fortunately, we have an explicit formula that
gives the number $T_{\III}(n)$ as follows 
(see \cite{NRS, VP1968, RS}):

\begin{equation}\label{Equ:NumTypeIIICodes}
	T_{\III}(n)
	= 
	\prod\limits_{i=0}^{\frac{n}{2} - 1}
	(3^{i}+1).
\end{equation}

\begin{lem}\label{Lem:NumSelfOrthoVecTypeIII}
	Let $n \equiv 0\pmod 4$. 
	Then the number of self-orthogonal vectors in~$\FF_{3}^{n}$
	is $3^{n-1} + 3^{\frac{n}{2}} - 3^{\frac{n}{2}-1}$.  
\end{lem}

\begin{proof}
	It is immediate from~\cite[Theorem 65]{Dickson}.
	However, we give a different proof for 
	this particular case. By Lemma~\ref{Lem:SlfOrthoTypeIII}, 
	the number of self-orthogonal vectors in~$\FF_{3}^{n}$
	for~$n \equiv 0\pmod 4$ is
	\[
		C(n,0) + 2^{3} C(n,3) + 2^{6} C(n,6) + \cdots +3^{n} C(n,n),
	\]
	where $C(n,k)$ is the binomial function. 
	Now let $\omega$ be the primitive cube root of unity,
	satisfying $\omega^3 = 1$ and $1 + \omega+\omega^2 = 0$.
	Then immediately we can have
	\begin{align*}
		&C(n,0) + 2^{3} C(n,3) + 2^{6} C(n,6) + \cdots +2^{n} C(n,n)\\
		& =
		\frac{3^n + (1+2\omega)^{n} + (1 + 2\omega^2)^{n}}{3}\\
		& =
		\frac{3^{n} + 2.3^{\frac{n}{2}}}{3}.
	\end{align*}
	Hence, the number of 
	self-orthogonal vectors 
	is $3^{n-1} + 3^{\frac{n}{2}} - 3^{\frac{n}{2}-1}$. 
\end{proof}

\begin{thm}\label{Thm:NeighborTypeIII}
	Let~$C$ be a Type~$\III$ code of length~$n \equiv 0 \pmod 4$.
	Then $N_{C}(v)$ is a Type~$\III$ code if and only if $v \in \FF_{3}^{n}$ is a self-orthogonal vector not in~$C$. 
%	such that $\wt(v) \equiv 0\pmod 3$.
\end{thm}

\begin{proof}
	Suppose $N_{C}(v)$ is a Type~$\III$ code.
	Then $v$ must be a self-orthogonal vector, 
	otherwise $N_{C}(v)$ will no longer a Type~$\III$ code.
	
	Conversely, suppose that $v \in \FF_{3}^{n}$ is a self-orthogonal vector not in~$C$. 
	Then by Lemma~\ref{Lem:SlfOrthoTypeIII},
	$\wt(v) \equiv 0\pmod 3$.
	Since $C$ is a Type~$\III$ code, therefore 
	$C_{0} = \{w \in C \mid w \cdot v = 0\}$
	is a subcode of~$C$ with co-dimension~$1$.
	Let $w \in C_{0}$.
	Then 
	$\wt(w + v) \equiv 0 \pmod 3$, 
	since self-orthogonal vectors $w$ and $v$ 
	are orthogonal to each other and 
	\begin{align*}
		(w + v)
		\cdot
		(w + v)
		& =
		w\cdot w
		+
		2 w\cdot v
		+
		v\cdot v
		= 0.
	\end{align*}
	Therefore, the weight of each vector in~$N_{C}(v)$ 
	is a multiple of~$3$ and hence $N_{C}(v)$
	is a Type~$\III$ code. 
\end{proof}

\begin{thm}\label{Thm:ConnectTypeIII}
	Let $n \equiv 0 \pmod 4$. Let $C$ be a Type~$\III$ code of length~$n$.
	If $C^{\prime} = N_{C}(v)$ for some self-orthogonal vector~$v \in \FF_{3}^{n}$, then $C = N_{C^{\prime}}(w)$
	for some self-orthogonal vector~$w \in \FF_{3}^{n}$.
\end{thm}

\begin{proof}
	Let $C$ be a Type~$\III$ code
	of length~$n \equiv 0\pmod 4$.
	Let $v \in \FF_{3}^{n}$ be a self-orthogonal 
	vector not in~$C$.
	Then by Lemma~\ref{Lem:SlfOrthoTypeIII},
	we get $\wt(v) \equiv 0 \pmod 3$.
	Then 
	$C_{0} = \{u \in C \mid u \cdot v = 0\}$
	is a subcode of~$C$ with co-dimension~$1$.
	This implies $C = \langle C_{0}, w \rangle$ 
	for some self-orthogonal vector~$w$. 
	By Lemma~\ref{Lem:SlfOrthoTypeIII}, $\wt(w) \equiv 0\pmod 3$.
	Clearly, $w$ is orthogonal to each vector in~$C_{0}$.
	Moreover, by Theorem~\ref{Thm:NeighborTypeIII}, 
	we have
	$C^{\prime} = N_{C}(v)$ is a Type~$\III$ code.
	This implies $N_{C^{\prime}}(w)$ is also a Type~$\III$
	code. 
	Let $C_{0}^{\prime} = \{u^{\prime} \in C^{\prime} \mid u^{\prime}\cdot w = 0\}$.
	Then $C_{0}^{\prime}$ is a subcode of~$C^{\prime}$ with co-dimension~$1$.
	Since $w$ is orthogonal to each vector in~$C_{0}$,
	therefore  $C_{0}^{\prime} = C_{0}$.
	Hence $C = \langle C_{0}, w \rangle = \langle C_{0}^{\prime}, w \rangle = N_{C^{\prime}}(w)$.
\end{proof}

\begin{df}
	Let $n \equiv 0\pmod 4$.
	Let $V_{\III}(n)$ be the set of all Type~$\III$ codes 
	of length~$n$.
	Let $\Gamma_{\III}(n):=(V_{\III}(n),E_{\III}(n))$ be a graph, 
	where any two vertices in~$V_{\III}(n)$ are connected 
	by an edge in~$E_{\III}(n)$ if and only if they are neighbors. 
\end{df}

The following theorems gives basic properties of $\Gamma_{\III}(n)$ 
and answers the various counting questions related to it. 

\begin{thm}\label{Thm:NeighborGraphTypeIII}
	The graph $\Gamma_{\III}(n)$ is simple and undirected.
\end{thm}

\begin{proof}
	Since any Type~$\III$ code is not a neighbor of itself, therefore
	$\Gamma_{\III}(n)$ contains no loop. 
	Moreover, by Theorem~\ref{Thm:ConnectTypeIII}, 
	we have if $C$ is connected to~$C^{\prime}$, 
	then $C^{\prime}$ is connected to~$C$.
	This implies $\Gamma_{\III}(n)$ is not a directed graph.
\end{proof}

\begin{thm}\label{Thm:MaxDistanceTypeIII}
	Let $n \equiv 0\pmod 4$. 
	Then the graph $\Gamma_{\III}(n)$ is connected 
	with maximum path length $\frac{n}{2}$ 
	between two vertices.
\end{thm}

\begin{proof}
	Let $C_{1}$ and $C_{2}$ be two Type~$\III$ 
	codes of length~$n \equiv 0\pmod 4$.
	Let $C_{2} = \langle v_{1},\ldots, v_{\frac{n}{2}}\rangle$.
	Then each~$v_{i}$ is a self-orthogonal vector 
	such that $\wt(v_{i}) \equiv 0 \pmod 3$.
	Let $D_{1} := N_{C_{1}}(v_{1})$ 
	and $D_{i} := N_{D_{i-1}}(v_{i})$
	for $i = 2,\ldots, \frac{n}{2}$.
	Then $C_{1}$, $D_{1}$, $D_{2}$, $\ldots$, $D_{\frac{n}{2}} = C_{2}$
	is the path from $C_{1}$ to $C_{2}$.
	Hence the graph $\Gamma_{\III}(n)$ is connected. 
	By %Lemma~\ref{Lem:SelfDualCodes}
	Example~\ref{Ex:MaxPathLength}, 
	we can have two Type~$\III$ codes, say $C_{1}^{\prime}$ 
	and $C_{2}^{\prime}$ such that the maximum path length 
	between them in~$\Gamma_{\III}(4)$ is $2$.
	Now let the following $k$-times direct sums for positive integer~$k$:
	\begin{align*}
		C_{1} & = C_{1}^{\prime} \oplus \cdots \oplus C_{1}^{\prime},\\
		C_{2} & = C_{2}^{\prime} \oplus \cdots \oplus C_{2}^{\prime}.
	\end{align*}
	This implies the length of $C_{1}$ and $C_{2}$
	is $4k$ and $C_{1}\cap C_{2} = \0$. Hence there exists two Type~$\III$ 
	codes of length $n \equiv 0\pmod 4$ such that
	the maximum path length is $\frac{n}{2}$ in 
	the graph~$\Gamma_{\III}(n)$ is $\frac{n}{2}$.
\end{proof} 

\begin{ex}\label{Ex:MaxPathLength}
	Let $C_{1}$ be a code of length~$4$ over~$\FF_{3}$ with generator matrix:
	\[
		\begin{pmatrix}
			1 & 0 & 1 & 1\\
			0 & 1 & 1 & 2
		\end{pmatrix}.
	\]
	It is easy to check that $C_{1}$ is a Type~$\III$ code.
	Then $D_{1} := N_{C_{1}}(v_{1})$ is a neighbor of~$C_{1}$,
	where $v_{1} = (1,0,2,2)$ is a self-orthogonal 
	vector in $\FF_{3}^{4}$ and not in~$C_{1}$.
	Also, $D_{2} := N_{D_{1}}(v_{2})$ is a neighbor of~$D_{1}$,
	where $v_{2} = (0,1,2,1)$ is a self-orthogonal 
	vector in $\FF_{3}^{4}$ and not in~$D_{1}$.
	Immediately, $D_{1}$ and $D_{2}$ are Type~$\III$ codes.
	The generator matrix of $D_{2}$ is:
	\[
	\begin{pmatrix}
		1 & 0 & 2 & 2\\
		0 & 1 & 2 & 1
	\end{pmatrix}.
	\]
	Moreover, we can see that $C_{1} \cap D_{2} = \bm{0}$.
	This conclude that the maximum path length 
	of the graph $\Gamma_{\III}(4)$ is $2$. 
\end{ex}

\begin{thm}\label{Thm:GrVertexNumTypeIII}
	Let $n \equiv 0\pmod 4$.
	Then the number of vertices in $\Gamma_{\III}(n)$ is
	$\prod_{i=0}^{\frac{n}{2} - 1}(3^{i}+1)$.
\end{thm}

\begin{proof}
	We can have the number of Type~$\III$ codes of length~$n\equiv 0\pmod 4$ from~(\ref{Equ:NumTypeIIICodes}).
	This completes the proof.
\end{proof}

\begin{lem}\label{Lem:NumEqualNeighborTypeIII}
	Let $n \equiv 0 \pmod 4$. 
	Let~$C$ be a Type~$\III$ code of length~$n$.
	Suppose~$v_{0} \in \FF_{3}^{n}$ 
	be a self-orthogonal vector not in~$C$.
%	such that~$\wt(v_{0}) \equiv 0\pmod 3$.
	Then the number of self-orthogonal vectors~$v \in \FF_{3}^{n}$
	such that $N_{C}(v) = N_{C}(v_{0})$
	is $2.3^{\frac{n}{2}-1}$.
\end{lem}

\begin{proof}
	By Lemmas~\ref{Lem:SlfOrthoTypeIII} and \ref{Lem:NumEqualNeighbor},
	we can obtain the result.
\end{proof}

\begin{thm}\label{Thm:GrDegTypeIII}
	Let $n \equiv 0\pmod 4$.
	Then the graph~$\Gamma_{\III}(n)$ is regular with degree~$\frac{1}{2}(3^{\frac{n}{2}} - 1)$.	
\end{thm}

\begin{proof}
	Let $C$ be a Type~$\III$ code of length~$n \equiv 0\pmod 4$.
	Then by Lemma~\ref{Lem:NumSelfOrthoVecTypeIII}, 
	we have the number of self-orthogonal vectors in~$\FF_{3}^{n}$ 
	but not in~$C$ is $3^{n-1}-3^{\frac{n}{2}-1}$.
	Moreover, by Lemma~\ref{Lem:NumEqualNeighborTypeIII}, 
	each Type~$\III$ code of length~$n$ occurs
	$2.3^{\frac{n}{2}-1}$ times. 
	Hence the degree of each vertex~$v$ in $\Gamma_{\III}(n)$
	is
	\begin{align*}
		\deg_{\Gamma_{\III}(n)}(v)
		=
		\frac{3^{n-1}-3^{\frac{n}{2}-1}}{2.3^{\frac{n}{2}-1}}
		=
		\frac{3^{\frac{n}{2}}-1}{2}.
	\end{align*}
\end{proof}

\begin{thm}\label{Thm:GrNeighborEdgeNum}
	Let $n \equiv 0\pmod 4$.
	Then the number of edges in~$\Gamma_{\III}(n)$ is
	$$\frac{1}{2}\left( 
	\prod_{i=1}^{\frac{n}{2} - 1}
	(3^{i}+1)
	\right)
	(3^{\frac{n}{2}} - 1).$$
\end{thm}

\begin{proof}
	By Theorem~\ref{Thm:GrVertexNumTypeIII}, the number of vertex
	in the graph~$\Gamma_{\III}(n)$ 
	is~$2\prod_{i=1}^{\frac{n}{2} - 1}(3^{i}+1)$.
	Since the graph is regular with degree~$\frac{1}{2}(3^{\frac{n}{2}}-1)$. 
	Therefore,
	\[
		2|E_{\III}(n)|
		=
		\left(
		2 
		\prod_{i=1}^{\frac{n}{2} - 1}
		(3^{i}+1)
		\right)
		\frac{1}{2}(3^{\frac{n}{2}} - 1).
	\]
	This gives the result.
\end{proof}

%\begin{thm}
%	Let $n \equiv 0\pmod {12}$.
%	Then the graph~$\Gamma_{\III}(n)$ contains an Euler cycle.
%\end{thm}

%\begin{proof}
%	content...
%\end{proof}

In the graph~$\Gamma_{\III}(n)$, 
if the shortest path between two vertices has length~$k$,
we call the two vertices are in distance~$k$ apart. 
In this case, the corresponding two Type~$\III$ codes 
in~$V_{\III}(n)$ are called $k$-neighbors and
share a subcode of co-dimension~$k$. 
%For a self-dual code~$C \in V_{\III}(n)$,
%we denote by~$V_{q}^{k}(C)$ the set of all $k$-neighbors
%of~$C$ in~$\Gamma_{q}(n)$.

\begin{rem}\label{Rem:ZeroNeighborTypeIII}
	Every Type~$\III$ code in~$V_{\III}(n)$ is its $0$-neighbor.
\end{rem}

Let $C$ be a Type~$\III$ code of length~$n \equiv 0 \pmod 4$. 
For any non-negative integer~$k$, 
we denote the number of Type~$\III$ $k$-neighbors of~$C$ 
by $L_{k}^{\III}(n)$
in $\Gamma_{\III}(n)$. 
By Remark~\ref{Rem:ZeroNeighborTypeIII},
we have $L_{0}^{\III}(n) = 1$.
Now the following theorem gives 
$L_{k}^{\III}(n)$ 
for $k > 0$.

\begin{thm}\label{Thm:k-neighbors}
	Let $n \equiv 0 \pmod 4$. 
	Let $C$ be a Type~$\III$ code of length~$n$. 
	Then for $k > 0$, we have
	\[
	L_{k}^{\III}(n)
	=
	\dfrac{\prod_{i=0}^{k - 1} (3^{n-1-i}-3^{\frac{n}{2}-1})}{\prod_{j = 0}^{k-1}(3^{\frac{n}{2}}-3^{\frac{n}{2}-1-j})}.
	\]
\end{thm}

\begin{proof}
	Let $C$ be a Type~$\III$ code of length~$n$.
	Then
	\[
	L_{k}^{\III}(n)
	=
	\#
	\{
	D \in V_{\III}(n)
	\mid
	\text{$D$ is a $k$-neighbor of~$C$}
	\}.
	\]
	Let $S_{k}(n)$ be the set of~$k$ different self-orthogonal vectors
	$v_{1}, v_{2},\ldots, v_{k} \in \FF_{3}^{n}$
	and not in~$C$
	such that each~$v_{j}$ are orthogonal 
	to~$v_{1}, v_{2},\ldots, v_{j-1}$.
	Then
	\[
	L_{k}^{\III}(n)
	=
	\frac{\# S_{k}(n)}{\# \text{ways each neighbor of~$C$ is generated}}.
	\]
	By Lemma~\ref{Lem:NumSelfOrthoVecTypeIII}, we have the number of self-orthogonal vectors in~$\FF_{3}^{n}$ that are not in~$C$ is
	$3^{n-1} - 3^{\frac{n}{2}-1}$.
	Moreover, each choice of a self-orthogonal vector~$v_{j}$
	reduces the number of available self-orthogonal vectors 
	in ambient space by~$\frac{1}{3}$,
	since it must be orthogonal to the previous~$v_{j}$
	and its weight is multiple of~$3$.
	This provides that the number of choices for the vectors is
	$\prod_{i=0}^{k-1}(3^{n-1-i} - 3^{\frac{n}{2}-1})$
	the number of choices for self-orthogonal vectors.
	By using Lemma~\ref{Lem:NumEqualNeighbor} recursively,
	we can have~$\prod_{j = 0}^{k-1}(3^{\frac{n}{2}}-3^{\frac{n}{2}-1-j})$
	the number of ways each neighbor of~$C$ is generated.
	Hence
	\[
	L_{k}^{\III}(n)
	=
	\dfrac{\prod_{i=0}^{k - 1} (3^{n-1-i}-3^{\frac{n}{2}-1})}{\prod_{j = 0}^{k-1}(3^{\frac{n}{2}}-3^{\frac{n}{2}-1-j})}.
	\]
	This completes the proof.
\end{proof}

\begin{rem}
	Taking $k = 1$ in the above theorem, we have
	\[
	L_{1}^{\III}(n)
	=
	\frac{3^{n-1}-3^{\frac{n}{2}-1}}{3^{\frac{n}{2}}-3^{\frac{n}{2}-1}}
	=
	\frac{3^{n}-3^{\frac{n}{2}}}{2.3^{\frac{n}{2}}}
	=
	\frac{3^{\frac{n}{2}}-1}{2}.
	\]
	This gives the number of $1$-neighbors of~$C$ 
	as presented in Theorem~\ref{Thm:GrDegTypeIII}
\end{rem}

\begin{ex}
	Let $n = 4$. Then $T_{\III}(n) = \prod_{i=0}^{\frac{n}{2}-1} (3^{i}+1) = 8$.
	Moreover, $\deg(\Gamma_{\III}(n)) = \frac{1}{2}(3^{\frac{n}{2}}-1) =  4$.
	This implies the graph~$\Gamma_{\III}(n)$ has~$8$ vertices and is regular with degree~$4$. By Remark~\ref{Rem:ZeroNeighborTypeIII}, 
	we have $L_{k}^{\III}(n) = 1$ for $k = 0$. 
	By Theorem~\ref{Thm:k-neighbors},
	we have the following $k$-neighbors of $\Gamma_{\III}(n)$.
	
	\begin{tabular}{ll}
		For $k =1$:
		&
		$L_{k}^{\III}(n)
		=
		\dfrac{3^{3}-3}{3^2-3}
		= 
		4.$\\
		For $k =2$:
		&
		$L_{k}^{\III}(n)
		=
		\dfrac{(3^{3}-3)(3^2-3)}{(3^2-3)(3^2-1)}
		= 
		3.$
	\end{tabular}
	
	\noindent
	Then $1+4+3 = 8$ which is the total number of Type~$\III$ codes.
\end{ex}

The observation in the above example concludes the following result.

\begin{thm}\label{Thm:NumSlfDualkNeighbor}
	Let~$n \equiv 0\pmod 4$.
	Then the number of Type~$\III$ codes of length~$n$
	is $\sum_{k = 0}^{\frac{n}{2}} L_{k}^{\III}(n)$.
\end{thm}

\begin{proof}
	Let $C_{1}$ and $C_{2}$ be any two 
	Type~$\III$ codes of length~$n \equiv 0\pmod 4$.
	If $C_{1}$ and $C_{2}$ are connected by a path in~$\Gamma_{\III}(n)$,
	then by Theorem~\ref{Thm:MaxDistanceTypeIII}, 
	the maximum path length will be~$\frac{n}{2}$.
	This completes the proof. 	
\end{proof}

\begin{ex}
	Let $n = 0\pmod 4$ be the length of the Type~$\III$ codes,
	$|V_{\III}(n)|$ the number of vertices in the graphs~$\Gamma_{\III}(n)$.
	In Table~1
	%\ref{Tab:kNeighborTypeIII},
	we listed the $k$-neighbors of~$\Gamma_{\III}(n)$ up to~$n = 12$. 
	Note that in each row $k$ goes from $0$ to $\frac{n}{2}$ and sum
	the sum of the $k$-neighbors in each row is~$|V_{\III}(n)|$ as 
	in Theorem~\ref{Thm:NumSlfDualkNeighbor}. 
	
	\begin{table}[h!]
		\begin{center}
			\label{Tab:kNeighborTypeIII}
			\caption{List of $k$-Neighbors in $\Gamma_{\III}(n)$ up to~$n = 12$}
			\begin{tabular}{|c|c||c|c|c|c|c|c|c|}
				\hline
				\multirow{2}{*}{$n$} & \multirow{2}{*}{$|V_{\III}(n)|$} & \multicolumn{7}{c|}{$k$-neighbors}\\ \cline{3-9}
				 &  & 0 & 1 & 2 & 3 & 4 & 5 & 6 \\
				 \hline
				 4 & 8 & 1 & 4 & 3 & 0 & 0 & 0 & 0\\
				 8 & 2240 & 1 & 40 & 390 & 1080 & 729 & 0 & 0 \\
				 12 & 44817920 & 1 & 364 & 33033 & 914760 & 8027019 & 21493836 & 14348907 \\
				\hline\hline
			\end{tabular}
		\end{center}
	\end{table}
\end{ex}

\section{Type~$\III$ code with all-ones vector}\label{Sec:TypeIIICodesAllOne}

In this section, we assume the Type~$\III$ codes
that contain the all-ones vector~$\mathbf{1}$.
This additional assumption in Type~$\III$
codes conclude that the length of the code~$n\equiv 0\pmod {12}$.

\begin{lem}\label{Lem:NumSelfOrthoVecTypeIIIAllOne}
	Let $n \equiv 0\pmod {12}$. 
	The number of self-orthogonal vectors in~$\FF_{3}^{n}$
	that are also orthogonal to~$\allone$
	is $3^{n-2} + 3^{\frac{n}{2}} - 3^{\frac{n}{2}-1}$. 
\end{lem}

\begin{proof}
	Since $n \equiv 0\pmod {12}$. 
	By Lemma~\ref{Lem:NumSelfOrthoVecTypeIII}, we have 
	the number of self-orthogonal vectors in $\FF_{3}^{n}$ 
	is $3^{n-1} + 3^{\frac{n}{2}} - 3^{\frac{n}{2}-1}$.
	Then clearly the number of a self-orthogonal vector
	that are also orthogonal to~$\allone$ is 
	$3^{n-2} + 3^{\frac{n}{2}} - 3^{\frac{n}{2}-1}$.
\end{proof}

\begin{thm}\label{Thm:NeighborTypeIIIAllOne}
	Let $n \equiv 0\pmod {12}$.
	Let~$C$ be a Type~$\III$ code 
	of length~$n$
	containing all-ones vector.
	Then $N_{C}(v)$ is a Type~$\III$ code containing all-ones 
	vector if and only if~$v \in \FF_{3}^{n}$ is a self-orthogonal vector not in~$C$ such that 
%	$v \in \langle \mathbf{1} \rangle^{\perp}$
	$\allone \cdot v = 0$.
\end{thm}

\begin{proof}
	Suppose $N_{C}(v)$ is a Type~$\III$ code containing~$\bm{1}$.
	Then $v$ must be a self-orthogonal vector such that
%	$v \in \langle \bm{1} \rangle^{\perp}$
	$\allone \cdot v = 0$, 
	otherwise $N_{C}(v)$
	will no longer a Type~$\III$ code containing~$\bm{1}$.
	
	Conversely, suppose that $v \in \FF_{3}^{n}$ is a self-orthogonal vector not in~$C$ such that 
	$\allone \cdot v = 0$.
	This implies $v \in \langle \bm{1} \rangle^{\perp}$.
	Let $C_{0} = \{w \in C \mid w \cdot v = 0\}$.
	Then by Lemma~\ref{Lem:SlfOrthoTypeIII},
	$\wt(w + v) \equiv 0 \pmod 3$.
	Since $C$ is a Type~$\III$ code, 
	therefore $C_{0}$ is a subcode of~$C$ with co-dimension~$1$. 
	This implies that the weight of each vector in~$N_{C}(v)$ 
	is a multiple of~$3$.
	Moreover, $C$ contains~$\bm{1}$ and 
%	$v \in \langle \bm{1} \rangle^{\perp}$
	$\allone \cdot v = 0$.  
	Hence $N_{C}(v)$ is a Type~$\III$ code containing all-ones vector. 
\end{proof}

\begin{thm}\label{Thm:ConnectTypeIIIAllOne}
	Let $n \equiv 0 \pmod {12}$. Let $C$ be a Type~$\III$ code of length~$n$
	containing all-ones vector.
	If $C^{\prime} = N_{C}(v)$ for some self-orthogonal vector~$v$
	such that 
%	$v\in \langle \bm{1} \rangle^{\perp}$
	$\allone \cdot v = 0$, 
	then $C = N_{C^{\prime}}(w)$
	for some self-orthogonal vector~$w$
	such that 
	$\allone \cdot w = 0$
%	$w\in \langle \bm{1} \rangle^{\perp}$.
\end{thm}

\begin{proof}
	Let $C$ be a Type~$\III$ code containing~$\bm{1}$
	and $v$ be a self-orthogonal 
	vector not in~$C$ such that 
	$\allone \cdot v = 0$.
	This implies
	$v\in \langle \bm{1} \rangle^{\perp}$.
	Then 
	$C_{0} = \{u \in C \mid u \cdot v = 0\}$
	is a subcode of~$C$ with co-dimension~$1$.
	This implies $C = \langle C_{0}, w \rangle$ 
	for some self-orthogonal vector~$w$ 
	such that 
%	$w\in \langle \bm{1} \rangle^{\perp}$
	$\allone \cdot w = 0$.
	Clearly, $w$ is orthogonal to each vector in~$C_{0}$.
	By Lemma~\ref{Lem:SlfOrthoTypeIII}, we have
	$\wt(v) \equiv 0\pmod 4$ and
	$\wt(w) \equiv 0\pmod 4$.
	Moreover, by~Theorem~\ref{Thm:NeighborTypeIIIAllOne}, 
	$C^{\prime} = N_{C}(v)$ is a Type~$\III$ code
	containing~$\bm{1}$.
	This implies $N_{C^{\prime}}(w)$ is also a Type~$\III$
	code. 
	Let $C_{0}^{\prime} = \{u^{\prime} \in C^{\prime} \mid u^{\prime}\cdot w = 0\}$.
	Then $C_{0}^{\prime}$ is a subcode of~$C^{\prime}$ with co-dimension~$1$.
	Since $w$ is orthogonal to each vector in~$C_{0}$,
	therefore  $C_{0}^{\prime} = C_{0}$.
	Hence $C = \langle C_{0}, w \rangle = \langle C_{0}^{\prime}, w \rangle = N_{C^{\prime}}(w)$.
\end{proof}

\begin{df}
	Let $n \equiv 0\pmod {12}$.
	Let $V_{\III}(n,\allone)$ be the set all Type~$\III$ codes 
	of length~$n$ containing~$\allone$.
	Let $\Gamma_{\III}(n,\allone):=(V_{\III}(n,\allone),E_{\III}(n,\allone))$ 
	be a graph, where any two vertices in~$V_{\III}(n,\allone)$ are connected 
	by an edge in~$E_{\III}(n,\allone)$ if and only if they are neighbors. 
\end{df}

The following theorems gives basic properties of $\Gamma_{\III}(n)$ 
and answers the various counting questions related to it. 

\begin{thm}\label{Thm:NeighborGraphTypeIIIAllOne}
	Let $n \equiv 0\pmod{12}$.
	Then the graph $\Gamma_{\III}(n,\allone)$ is simple and undirected.
\end{thm}

\begin{proof}
	Since any Type~$\III$ code containing~$\allone$ 
	is not a neighbor of itself, therefore
	$\Gamma_{\III}(n,\allone)$ contains no loop. 
	Moreover, by Theorem~\ref{Thm:ConnectTypeIIIAllOne}, 
	we have if $C$ is connected to~$C^{\prime}$, 
	then $C^{\prime}$ is connected to~$C$.
	This implies $\Gamma_{\III}(n,\allone)$ is not a directed graph.
\end{proof}

\begin{thm}\label{Thm:MaxDistanceTypeIIIAllOne}
	Let $n \equiv 0\pmod{12}$.
	The graph $\Gamma_{\III}(n,\allone)$ is connected 
	with maximum path length $\frac{n}{2}-1$ 
	between two vertices.
\end{thm}

\begin{proof}
	Let $C_{1}$ and $C_{2}$ be two Type~$\III$ codes 
	of length~$n \equiv 0\pmod{12}$ containing~$\allone$.
	Let $C_{2} = \langle v_{1},\ldots, v_{\frac{n}{2}}\rangle$.
	Then each~$v_{i}$ is a self-orthogonal vector such that
	$\allone \cdot v_{i} = \0$.
	This implies each
	$\bm{v_{i}}\in \langle \bm{1} \rangle^{\perp}$.
	Let $D_{1} = N_{C_{1}}(v_{1})$ 
	and $D_{i} = N_{D_{i-1}}(v_{i})$
	for $i = 2,\ldots, \frac{n}{2}$.
	Then $C_{1}$, $D_{1}$, $D_{2}$, $\ldots$, $D_{\frac{n}{2}}$
	is the path from $C_{1}$ to $C_{2}$.
	Hence the graph $\Gamma_{\III}(n,\allone)$ is connected. 
	Since each Type~$\III$ code in $\Gamma_{\III}(n,\allone)$
	contains~$\allone$, 
	they must have a subspace of dimension~$1$ 
	in common. 
	This implies that the maximum distance in 
	the graph $\Gamma_{\III}(n,\allone)$ is $\frac{n}{2}-1$.
\end{proof}

\begin{thm}\label{Thm:GrVertexNumTypeIIIAllOne}
	Let $n \equiv 0\pmod{12}$.
	Then the number of vertices in $\Gamma_{\III}(n,\allone)$ is
	$2\prod_{i=1}^{\frac{n}{2} - 2}(3^{i}+1)$.
\end{thm}

\begin{proof}
	We know that the number of Type~$\III$ codes of 
	length~$n \equiv 0 \pmod {12}$ containing~$\allone$
	is $2\prod_{i=1}^{\frac{n}{2} - 2}(3^{i}+1)$,
	see~\cite{PP1973}.
	This completes the proof.
\end{proof}

\begin{lem}\label{Lem:NumEqualNeighborTypeIIIAllOne}
	Let $n \equiv 0 \pmod {12}$. 
	Let~$C$ be a Type~$\III$ code of length~$n$
	containing~$\bm{1}$.
	Suppose~$v_{0} \in \FF_{3}^{n}$ 
	be a self-orthogonal vector not in~$C$
	such that~$\allone \cdot v_{0} = \0$.
%	$v_{0} \in \langle \allone \rangle^{\perp}$.
	Then the number of self-orthogonal vectors~$v \in \FF_{3}^{n}$
	such that $N_{C}(v) = N_{C}(v_{0})$
	is $2.3^{\frac{n}{2}-1}$.
\end{lem}

\begin{proof}
	By Lemmas~\ref{Lem:SlfOrthoTypeIII}
	and \ref{Lem:EqualNeighbors},
	we can obtain the result.
\end{proof}

\begin{thm}\label{Thm:GrDegTypeIIIAllOne}
	Let $n \equiv 0\pmod {12}$.
	Then the graph~$\Gamma_{\III}(n,\allone)$ is regular with degree~$\dfrac{1}{2}(3^{\frac{n}{2}-1} - 1)$.	
\end{thm}

\begin{proof}
	Let $C$ be a Type~$\III$ code containing~$\allone$.
	Then by Lemma~\ref{Lem:NumSelfOrthoVecTypeIIIAllOne}, 
	we have the number of self-orthogonal vectors in~$\FF_{3}^{n}$ 
	that are orthogonal to $\allone$ but not in~$C$ is $3^{n-2}-3^{\frac{n}{2}-1}$.
	Moreover, by Lemma~\ref{Lem:NumEqualNeighborTypeIIIAllOne}, 
	each Type~$\III$ code of length~$n$ occurs
	$2.3^{\frac{n}{2}-1}$ times. 
	Hence the degree of each vertex~$v$ in $\Gamma_{\III}(n,\allone)$
	is
	\begin{align*}
		\deg_{\Gamma_{\III}(n,\allone)}(v)
		=
		\frac{3^{n-2}-3^{\frac{n}{2}-1}}{2.3^{\frac{n}{2}-1}}
		=
		\frac{3^{\frac{n}{2}-1}-1}{2}.
	\end{align*}
\end{proof}

\begin{thm}\label{Thm:GrNeighborEdgeNumAllOne}
	Let $n \equiv 0\pmod {12}$.
	Then the number of edges in the graph~$\Gamma_{\III}(n,\allone)$ is
	$$\dfrac{1}{2}\left( 
	\prod_{i=1}^{\frac{n}{2} - 2}
	(3^{i}+1)
	\right)
	(3^{\frac{n}{2}-1} - 1)$$.
\end{thm}

\begin{proof}
	By Theorem~\ref{Thm:GrVertexNumTypeIIIAllOne}, the number of vertex
	in the graph~$\Gamma_{\III}(n,\allone)$ 
	is~$2\prod_{i=1}^{\frac{n}{2} - 2}(3^{i}+1)$.
	Since the graph is regular with degree~$\frac{1}{2}(3^{\frac{n}{2}-1}-1)$. 
	Therefore,
	\[
	2|E_{\III}(n,\allone)|
	=
	\left(
	2 
	\prod_{i=1}^{\frac{n}{2} - 2}
	(3^{i}+1)
	\right)
	\frac{1}{2}(3^{\frac{n}{2}-1} - 1).
	\]
	This gives the result.
\end{proof}

%\begin{thm}
%	The graph~$\Gamma_{\III}(n,\allone)$ contains an Euler cycle.
%\end{thm}

\begin{rem}\label{Rem:ZeroNeighborTypeIIIAllOne}
	Every Type~$\III$ code in~$V_{\III}(n,\allone)$ is its $0$-neighbor.
\end{rem}

Let $C$ be a Type~$\III$ code of length~$n \equiv 0 \pmod {12}$
containing~$\allone$. 
For any non-negative integer~$k$, 
we denote the number of Type~$\III$ $k$-neighbors of~$C$ 
by $L_{k}^{\III}(n,\allone)$
in $\Gamma_{\III}(n,\allone)$. 
By Remark~\ref{Rem:ZeroNeighborTypeIIIAllOne},
we have $L_{0}^{\III}(n,\allone) = 1$.
Now the following theorem gives 
$L_{k}^{\III}(n,\allone)$ 
for $k > 0$.

\begin{thm}\label{Thm:k-neighborsAllOne}
	Let $n \equiv 0 \pmod {12}$. 
	Let $C$ be a Type~$\III$ code of length~$n$
	containing~$\allone$. 
	Then for $k > 0$, we have
	\[
	L_{k}^{\III}(n,\allone)
	=
	\dfrac{\prod_{i=0}^{k - 1} (3^{n-2-i}-3^{\frac{n}{2}-1})}{\prod_{j = 0}^{k-1}(3^{\frac{n}{2}}-3^{\frac{n}{2}-1-j})}.
	\]
\end{thm}

\begin{proof}
	Let $C$ be a Type~$\III$ code of length~$n \equiv 0\pmod {12}$ 
	containing~$\allone$.
	Then
	\[
	L_{k}^{\III}(n,\allone)
	=
	\#
	\{
	D \in V_{\III}(n,\allone)
	\mid
	\text{$D$ is a $k$-neighbor of~$C$}
	\}.
	\]
	Let $S_{k}(n,\allone)$ be the set of~$k$ different self-orthogonal 
	vectors $v_{1}, v_{2},\ldots, v_{k} \in \FF_{3}^{n}$
	and not in~$C$
	such that each~$v_{j}$ are orthogonal 
	to~$\allone, v_{1}, v_{2},\ldots, v_{j-1}$.
	Then
	\[
	L_{k}^{\III}(n,\allone)
	=
	\frac{\# S_{k}(n,\allone)}{\# \text{ways each $k$-neighbor of~$C$ is generated}}.
	\]
	By Lemma~\ref{Lem:NumSelfOrthoVecTypeIIIAllOne}, 
	we have the number of self-orthogonal vectors in~$\FF_{3}^{n}$ 
	that are orthogonal to~$\allone$ and not in~$C$ is
	$3^{n-2} - 3^{\frac{n}{2}-1}$.
	Moreover, each choice of a self-orthogonal vector~$v_{j}$
	reduces the number of available self-orthogonal vectors 
	in ambient space by~$\frac{1}{3}$,
	since it must be orthogonal to the previous~$v_{j}$
	and its weight is multiple of~$3$. 
	This provides that the number of choices for the vectors is
	$\prod_{i=0}^{k-1}(3^{n-2-i} - 3^{\frac{n}{2}-1})$
	the number of choices for self-orthogonal vectors.
	By using Lemma~\ref{Lem:NumEqualNeighbor} recursively,
	we can have~$\prod_{j = 0}^{k-1}(3^{\frac{n}{2}}-3^{\frac{n}{2}-1-j})$
	the number of ways each $k$-neighbor of~$C$ is generated.
	Hence
	\[
	L_{k}^{\III}(n,\allone)
	=
	\dfrac{\prod_{i=0}^{k - 1} (3^{n-2-i}-3^{\frac{n}{2}-1})}{\prod_{j = 0}^{k-1}(3^{\frac{n}{2}}-3^{\frac{n}{2}-1-j})}.
	\]
	This completes the proof.
\end{proof}

\begin{rem}
	Taking $k = 1$ in the above theorem, we have
	\[
		L_{1}^{\III}(n,\allone)
		=
		\frac{3^{n-2}-3^{\frac{n}{2}-1}}{3^{\frac{n}{2}}-3^{\frac{n}{2}-1}}
		=
		\frac{3^{n-1}-3^{\frac{n}{2}}}{2.3^{\frac{n}{2}}}
		=
		\frac{3^{\frac{n}{2}-1}-1}{2}.
	\]
	This gives the number of $1$-neighbors of a Type~$\III$
	code with~$\allone$
	as presented in Theorem~\ref{Thm:GrDegTypeIIIAllOne}.
\end{rem}

\begin{ex}
	Let $n = 12$. 
	Then $T_{\III}(n,\allone) = \prod_{i=0}^{\frac{n}{2}-2} (3^{i}+1) = 183680$.
	Moreover, $\deg(\Gamma_{\III}(n)) = \frac{1}{2}(3^{\frac{n}{2}-1}-1) =  121$.
	This implies the graph~$\Gamma_{\III}(n,\allone)$ has~$183680$ vertices and is regular with degree~$121$. By Remark~\ref{Rem:ZeroNeighborTypeIIIAllOne}, 
	we have $L_{k}^{\III}(n) = 1$ for $k = 0$. 
	By Theorem~\ref{Thm:k-neighborsAllOne},
	we have the following $k$-neighbors of $\Gamma_{\III}(n,\allone)$.
	
	\begin{tabular}{ll}
		For $k =1$:
		&
		$L_{k}^{\III}(n,\allone)
		=
		\dfrac{3^{10}-3^5}{3^6-3^5}
		= 
		121.$\\
		For $k =2$:
		&
		$L_{k}^{\III}(n,\allone)
		=
		\dfrac{(3^{10}-3^{5})(3^9-3^5)}{(3^6-3^{5})(3^6-3^4)}
		= 
		3630.$\\
		For $k = 3$:
		&
		$L_{k}^{\III}(n,\allone)
		=
		\dfrac{(3^{10}-3^{5})(3^9-3^5)(3^8-3^5)}{(3^6-3^{5})(3^6-3^4)(3^6-3^3)}
		= 
		32670.$\\
		For $k = 4$:
		&
		$L_{k}^{\III}(n,\allone)
		=
		\dfrac{(3^{10}-3^{5})(3^9-3^5)(3^8-3^5)(3^7-3^5)}{(3^6-3^{5})(3^6-3^4)(3^6-3^3)(3^6-3^2)}
		= 
		88209.$\\
		For $k = 5$:
		&
		$L_{k}^{\III}(n,\allone)
		=
		\dfrac{(3^{10}-3^{5})(3^9-3^5)(3^8-3^5)(3^7-3^5)(3^6-3^5)}{(3^6-3^{5})(3^6-3^4)(3^6-3^3)(3^6-3^2)(3^6-3)}
		= 
		59049.$
	\end{tabular}
	
	\noindent
	Then $1+121+3630+32670+88209+59049 = 183680$ which is the total number of Type~$\III$ codes containing~$\allone$.
\end{ex}

Now we have the following analogous result of Theorem~\ref{Thm:NumSlfDualkNeighbor}.

\begin{thm}\label{Thm:NumSlfDualkNeighborAllOne}
	Let~$n \equiv 0\pmod {12}$.
	The number of Type~$\III$ codes of length~$n$
	containing~$\allone$
	is $\sum_{k = 0}^{\frac{n}{2}-1} L_{k}^{\III}(n,\allone)$.
\end{thm}

\begin{proof}
	Let $C_{1}$ and $C_{2}$ be any two 
	Type~$\III$ codes of length~$n \equiv 0\pmod {12}$
	containing~$\allone$.
	If $C_{1}$ and $C_{2}$ are connected by a path 
	in~$\Gamma_{\III}(n,\allone)$,
	then by Theorem~\ref{Thm:MaxDistanceTypeIIIAllOne}, 
	the maximum path length will be~$\frac{n}{2}-1$.
	This completes the proof. 	
\end{proof}

\section{Type~$\IV$ codes}\label{Sec:TypeIVCodes}

A self-dual code over~$\FF_4$ 
where all codewords have even weights 
is called \emph{Type~$\IV$} code.
It is well-known that the Type~$\IV$ code exists 
if and only if $n \equiv 0\pmod 2$.
Let $T_{\IV}(n)$ be the number of Type~$\IV$ codes of 
length~$n \equiv 0 \pmod 2$. 
From~\cite{NRS, VP1968, RS}, 
we have an explicit formula that
gives the number $T_{\IV}(n)$ as follows :

\begin{equation}\label{Equ:NumTypeIVCodes}
	T_{\IV}(n)
	= 
	\prod\limits_{i=0}^{\frac{n}{2} - 1}
	(2^{2i+1}+1).
\end{equation}

\begin{lem}\label{Lem:NumSelfOrthoVecTypeIV}
	Let $n \equiv 0\pmod 2$. 
	Then the number of self-orthogonal vectors in~$\FF_{4}^{n}$
	is $2^{n-1}(2^n+1)$. 
\end{lem}

\begin{proof}
	By Lemma~\ref{Lem:SlfOrthoTypeIV}, 
	the number of self-orthogonal vectors of 
	length~$n \equiv 0\pmod 2$ is
	\[
		C(n,0) + 3^{2} C(n,2) + 3^{4} C(n,4) + \cdots +3^{n} C(n,n),
	\]
	where $C(n,k)$ is the binomial function. 
	Then immediately it can be written that
	 \begin{align*}
	 	&C(n,0) + 3^{2} C(n,2) + 3^{4} C(n,4) + \cdots +3^{n} C(n,n)\\
	 	& =
	 	\frac{4^n + 2^{n}}{2}\\
	 	& =
	 	2^{2n-1} + 2^{n-1}.
	 \end{align*}
	Hence the number of self-orthogonal vectors is 
	$2^{n-1} (2^n+1)$. 
\end{proof}

\begin{thm}\label{Thm:NeighborTypeIV}
	Let $n \equiv 0\pmod 2$.
	Let~$C$ be a Type~$\IV$ code of length~$n$.
	Then $N_{C}(v)$ is a Type~$\IV$ code if and only if 
	$v \in \FF_{4}^{n}$ is a self-orthogonal vector. 
\end{thm}

\begin{proof}
	Suppose $N_{C}(v)$ is a Type~$\IV$ code of length~$n \equiv 0\pmod 2$.
	Then $v$ must be a self-orthogonal vector with even weight, 
	otherwise $N_{C}(v)$ will no longer a Type~$\IV$ code.
	
	Conversely, suppose that $v \in \FF_{4}^{n}$ 
	is a self-orthogonal vector not in~$C$.
	Then by Lemma~\ref{Lem:SlfOrthoTypeIV}, 
	$\wt(v) \equiv 0\pmod 2$. 
	Since $C$ is a Type~$\IV$ code, therefore 
	$C_{0} = \{w \in C \mid w \cdot v = 0\}$
	is a subcode of~$C$ with co-dimension~$1$.
	Let $w \in C_{0}$.
	Then 
	$\wt(w + v) \equiv 0 \pmod 2$, 
	since self-orthogonal vectors $w$ and $v$ 
	have even weights and are orthogonal to each other and 
	\begin{align*}
		(w + v)
		\cdot
		(w + v)
		& =
		w\cdot w
		+
		2 w\cdot v
		+
		v \cdot v
		= 0.
	\end{align*}
	Therefore, the weight of each vector in~$N_{C}(v)$ 
	is even and hence $N_{C}(v)$
	is a Type~$\IV$ code. 
\end{proof}

\begin{thm}\label{Thm:ConnectTypeIV}
	Let $n \equiv 0 \pmod 2$. 
	Let $C$ be a Type~$\IV$ code of length~$n$.
	If $C^{\prime} = N_{C}(v)$ 
	for some self-orthogonal vector~$v \in \FF_{4}^{n}$,
	then $C = N_{C^{\prime}}(w)$
	for some self-orthogonal vector~$w \in \FF_{4}^{n}$.
\end{thm}

\begin{proof}
	Let $C$ be a Type~$\IV$ code and $v$ be a self-orthogonal 
	vector not in~$C$. 
	Then by Lemma~\ref{Lem:SlfOrthoTypeIV},
	we have $\wt(v)$ is even.
	Then 
	$C_{0} = \{u \in C \mid u \cdot v = 0\}$
	is a subcode of~$C$ with co-dimension~$1$.
	This implies $C = \langle C_{0}, w \rangle$ 
	for some self-orthogonal vector~$w$.
	By Lemma~\ref{Lem:SlfOrthoTypeIV},
	$\wt(w) \equiv 0\pmod 2$.
	Clearly, $w$ is orthogonal to each vector in~$C_{0}$.
	Moreover, by Theorem~\ref{Thm:NeighborTypeIV}, 
	$C^{\prime} = N_{C}(v)$ is a Type~$\IV$ code.
	This implies $N_{C^{\prime}}(w)$ is also a Type~$\IV$
	code. 
	Let $C_{0}^{\prime} = \{u^{\prime} \in C^{\prime} \mid u^{\prime}\cdot w = 0\}$.
	Then $C_{0}^{\prime}$ is a subcode of~$C^{\prime}$ with co-dimension~$1$.
	Since $w$ is orthogonal to each vector in~$C_{0}$,
	therefore  $C_{0}^{\prime} = C_{0}$.
	Hence $C = \langle C_{0}, w \rangle = \langle C_{0}^{\prime}, w\rangle = N_{C^{\prime}}(w)$.
\end{proof}

\begin{df}
	Let $n \equiv 0\pmod 2$.
	Let $V_{\IV}(n)$ be the set all Type~$\IV$ codes 
	of length~$n$.
	Let $\Gamma_{\IV}(n):=(V_{\IV}(n),E_{\IV}(n))$ be a graph, 
	where any two vertices in~$V_{\IV}(n)$ are connected 
	by an edge in~$E_{\IV}(n)$ if and only if they are neighbors. 
\end{df}

The following theorems gives basic properties of $\Gamma_{\IV}(n)$ 
and answers the various counting questions related to it. 

\begin{thm}\label{Thm:NeighborGraphTypeIV}
	Let $n \equiv 0\pmod 2$.
	Then the graph $\Gamma_{\IV}(n)$ is simple and undirected.
\end{thm}

\begin{proof}
	Since any Type~$\IV$ code is not a neighbor of itself, therefore
	$\Gamma_{\IV}(n)$ contains no loop. 
	Moreover, by Theorem~\ref{Thm:ConnectTypeIV}, 
	we have if $C$ is connected to~$C^{\prime}$, 
	then $C^{\prime}$ is connected to~$C$.
	This implies $\Gamma_{\IV}(n)$ is not a directed graph.
\end{proof}

\begin{thm}\label{Thm:MaxDistanceTypeIV}
	Let $n \equiv 0\pmod 2$. 
	Then the graph $\Gamma_{\IV}(n)$ is connected 
	with maximum path length $\frac{n}{2}$ 
	between two vertices.
\end{thm}

\begin{proof}
	Let $C_{1}$ and $C_{2}$ be two Type~$\IV$ 
	codes of length~$n \equiv 0\pmod 2$.
	Let $C_{2} = \langle v_{1},\ldots, v_{\frac{n}{2}}\rangle$.
	Then each~$v_{i}$ is a self-orthogonal vector 
	such that $\wt(v_{i}) \equiv 0 \pmod 2$.
	Let $D_{1} := N_{C_{1}}(v_{1})$ 
	and $D_{i} := N_{D_{i-1}}(v_{i})$
	for $i = 2,\ldots, \frac{n}{2}$.
	Then $C_{1}$, $D_{1}$, $D_{2}$, $\ldots$, $D_{\frac{n}{2}} = C_{2}$
	is the path from $C_{1}$ to $C_{2}$.
	Hence the graph $\Gamma_{\IV}(n)$ is connected. 
	To show the maximum path length between two vertices
	is $\frac{n}{2}$, let n = 2.
	Let 
	$$C_{1}^{\prime} = \{(0,0),(1,1),(\omega,\omega),(\omega^2,\omega^2)\}$$
	be a code of length~$2$ over~$\FF_{4}$.
	It is easy to check that $C_{1}^{\prime}$ is Type~$\IV$.
	Then immediately we have 
	$$C_{2}^{\prime} = \{(0,0),(1,\omega^2),(\omega,1),(\omega^2,\omega)\},$$
	which is a neighbor of $C_{1}^{\prime}$ have path length~$1$.
	Now let the following $k$-times 
	direct sums for positive integer~$k$:
	\begin{align*}
		C_{1} & = C_{1}^{\prime} \oplus \cdots \oplus C_{1}^{\prime},\\
		C_{2} & = C_{2}^{\prime} \oplus \cdots \oplus C_{2}^{\prime}.
	\end{align*}
	This implies the length of $C_{1}$ and $C_{2}$
	is $2k$ and $C_{1}\cap C_{2} = \0$. Hence there exists two Type~$\IV$ 
	codes of length $n \equiv 0\pmod 2$ such that
	the maximum path length is $\frac{n}{2}$ in 
	the graph~$\Gamma_{\IV}(n)$ is $\frac{n}{2}$.
\end{proof} 

\begin{ex}\label{Ex:MaxPathLengthIV}
	Let $C$ be a code of length~$4$ over~$\FF_{4}$ with generator matrix:
	\[
	\begin{pmatrix}
		1 & 1 & 0 & 0\\
		0 & 0 & 1 & 1\\
	\end{pmatrix}.
	\]
	It is easy to check that $C$ is a Type~$\IV$ code.
	Then $D_{1} := N_{C}(v_{1})$ is a neighbor of~$C$,
	where $v_{1} = (1,\omega,0,0)$ is a self-orthogonal 
	vector in $\FF_{4}^{4}$ and not in~$C$.
	Also, $D_{2} := N_{D_{1}}(v_{2})$ is a neighbor of~$D_{1}$,
	where $v_{2} = (0,0,1,\omega^{2})$ is a self-orthogonal 
	vector in $\FF_{4}^{4}$ and not in~$D_{1}$.
	Immediately, $D_{1}$ and $D_{2}$ are Type~$\IV$ codes.
	The generator matrix of $D_{2}$ is:
	\[
	\begin{pmatrix}
		1 & \omega & 0 & 0\\
		0 & 0 & 1 & \omega^{2}
	\end{pmatrix}.
	\]
	Moreover, we can see that $C \cap D_{2} = \bm{0}$.
	This conclude that the maximum path length 
	of the graph $\Gamma_{\IV}(4)$ is $2$. 
\end{ex}

\begin{thm}\label{Thm:GrVertexNumTypeIV}
	Let $n \equiv 0\pmod 2$.
	Then the number of vertices in $\Gamma_{\IV}(n)$ is
	$\prod_{i=0}^{\frac{n}{2} - 1}(2^{2i+1}+1)$.
\end{thm}

\begin{proof}
	We can have the number of Type~$\IV$ codes of length~$n\equiv 0\pmod 2$ from~(\ref{Equ:NumTypeIVCodes}).
	This completes the proof.
\end{proof}

\begin{lem}\label{Lem:NumEqualNeighborTypeIV}
	Let $n \equiv 0 \pmod 2$. 
	Let~$C$ be a Type~$\IV$ code of length~$n$.
	Suppose~$v_{0} \in \FF_{4}^{n}$ 
	be a self-orthogonal vector not in~$C$.
	Then the number of self-orthogonal vectors~$v \in \FF_{4}^{n}$
	such that $N_{C}(v) = N_{C}(v_{0})$
	is $3.4^{\frac{n}{2}-1}$.
\end{lem}

\begin{proof}
	By Lemmas~\ref{Lem:SlfOrthoTypeIV} and~\ref{Lem:NumEqualNeighbor},
	we can obtain the result.
\end{proof}

\begin{thm}\label{Thm:GrDegTypeIV}
	Let $n \equiv 0\pmod 2$.
	Then the graph~$\Gamma_{\IV}(n)$ is regular 
	with degree~$\frac{2}{3}(2^{n} - 1)$.	
\end{thm}

\begin{proof}
	Let $C$ be a Type~$\IV$ code of length~$n \equiv 0\pmod 2$.
	Then by Lemma~\ref{Lem:NumSelfOrthoVecTypeIV}, 
	we have the number of self-orthogonal vectors in~$\FF_{4}^{n}$ 
	but not in~$C$ is $2^{2n-1}-2^{n-1}$.
	Moreover, by Lemma~\ref{Lem:NumEqualNeighborTypeIV}, 
	each Type~$\IV$ code of length~$n$ occurs
	$3.4^{\frac{n}{2}-1}$ times. 
	Hence the degree of each vertex~$v$ in $\Gamma_{\IV}(n)$
	is
	\begin{align*}
		\deg_{\Gamma_{\IV}(n)}(v)
		=
		\frac{2^{2n-1}-2^{n-1}}{3.4^{\frac{n}{2}-1}}
		=
		\frac{2(2^{n}-1)}{3}.
	\end{align*}
\end{proof}

\begin{thm}\label{Thm:GrNeighborEdgeNumIV}
	Let $n \equiv 0\pmod 2$.
	Then the number of edge in~$\Gamma_{\IV}(n)$ is
	$$\left( 
	\prod_{i=1}^{\frac{n}{2} - 1}(2^{2i+1}+1)
	\right)
	(2^{n} - 1).$$
\end{thm}

\begin{proof}
	By Theorem~\ref{Thm:GrVertexNumTypeIV}, the number of vertex
	in the graph~$\Gamma_{\IV}(n)$ 
	is~$2\prod_{i=0}^{\frac{n}{2} - 1}(2^{2i+1}+1)$.
	Since the graph is regular 
	with degree~$\frac{2}{3}(2^{n}-1)$. 
	Therefore,
	\[
	2|E_{\IV}(n)|
	=
	\left(
	3 
	\prod_{i=1}^{\frac{n}{2} - 1}
	(2^{2i+1}+1)
	\right)
	\frac{2}{3}(2^{n} - 1).
	\]
	This gives the result.
\end{proof}

In the graph~$\Gamma_{\IV}(n)$, 
if the shortest path between two vertices has length~$k$,
we call the two vertices are in distance~$k$ apart. 
In this case, the corresponding two Type~$\IV$ codes 
in~$V_{\IV}(n)$ are called $k$-neighbors and
share a subcode of co-dimension~$k$.

\begin{rem}\label{Rem:ZeroNeighborTypeIV}
	Every Type~$\III$ code in~$V_{\IV}(n)$ is its $0$-neighbor.
\end{rem}

Let $C$ be a Type~$\IV$ code of length~$n \equiv 0 \pmod 2$. 
For any non-negative integer~$k$, 
we denote the number of Type~$\IV$ $k$-neighbors of~$C$ 
by $L_{k}^{\IV}(n)$
in $\Gamma_{\IV}(n)$. 
By Remark~\ref{Rem:ZeroNeighborTypeIV},
we have $L_{0}^{\IV}(n) = 1$.
Now the following theorem gives 
$L_{k}^{\IV}(n)$ 
for $k > 0$.

\begin{thm}\label{Thm:k-neighborsIV}
	Let $n \equiv 0 \pmod 2$. 
	Let $C$ be a Type~$\IV$ code of length~$n$. 
	Then for $k > 0$, we have
	\[
	L_{k}^{\IV}(n)
	=
	\dfrac{\prod_{i=0}^{k - 1} (2^{2n-1-2i}-2^{n-1})}{\prod_{j = 0}^{k-1}(2^{n}-2^{n-2-2j})}.
	\]
\end{thm}

\begin{proof}
	Let $C$ be a Type~$\IV$ code of length~$n \equiv 0\pmod 2$.
	Then
	\[
	L_{k}^{\IV}(n)
	=
	\#
	\{
	D \in V_{\IV}(n)
	\mid
	\text{$D$ is a $k$-neighbor of~$C$}
	\}.
	\]
	Let $S_{k}(n)$ be the set of~$k$ different self-orthogonal vectors
	$v_{1}, v_{2},\ldots, v_{k} \in \FF_{4}^{n}$
	and not in~$C$
	such that each~$v_{j}$ are orthogonal 
	to~$v_{1}, v_{2},\ldots, v_{j-1}$.
	Then
	\[
	L_{k}^{\IV}(n)
	=
	\frac{\# S_{k}(n)}{\# \text{ways each neighbor of~$C$ is generated}}.
	\]
	By Lemma~\ref{Lem:NumSelfOrthoVecTypeIV}, we have the number of self-orthogonal vectors in~$\FF_{4}^{n}$ that are not in~$C$ is
	$2^{2n-1} - 2^{n-1}$.
	
	Moreover, each choice of a self-orthogonal vector~$v_{j}$
	reduces the number of available self-orthogonal vectors 
	in ambient space by~$\frac{1}{4}$,
	since it is even weight and must be orthogonal 
	to the previous~$v_{j}$. 
	This provides that the number of choices for the vectors is
	$\prod_{i=0}^{k-1}(2^{2n-1-2i} - 2^{n-1})$
	the number of choices for self-orthogonal vectors.
	By using Lemma~\ref{Lem:NumEqualNeighbor} recursively,
	we can have~$\prod_{j = 0}^{k-1}(2^{n}-2^{n-2-2j})$
	the number of ways each neighbor of~$C$ is generated.
	Hence
	\[
	L_{k}^{\IV}(n)
	=
	\dfrac{\prod_{i=0}^{k - 1} (2^{2n-1-2i}-2^{n-1})}{\prod_{j = 0}^{k-1}(2^{n}-2^{n-2-2j})}.
	\]
	This completes the proof.
\end{proof}

\begin{rem}
	Taking $k = 1$ in the above theorem, we have
	\[
	L_{1}^{\IV}(n)
	=
	\frac{2^{2n-1}-2^{n-1}}{2^{n}-2^{n-2}}
	=
	\frac{2^{2n-1}-2^{n-1}}{3.2^{n-2}}
	=
	\frac{2(2^{n}-1)}{3}.
	\]
	This gives the number of $1$-neighbors of
	Type~$\IV$ codes
	as presented in Theorem~\ref{Thm:GrDegTypeIV}
\end{rem}

\begin{ex}
	Let $n = 6$. Then $T_{\IV}(n) = \prod_{i=0}^{\frac{n}{2}-1} (2^{2i+1}+1) = 891$.
	Moreover, $\deg(\Gamma_{\IV}(n)) = \frac{2}{3}(2^{n}-1) =  42$.
	This implies the graph~$\Gamma_{\IV}(n)$ has~$891$ vertices and is regular with degree~$42$. By Remark~\ref{Rem:ZeroNeighborTypeIV}, 
	we have $L_{k}^{\IV}(n) = 1$ for $k = 0$. 
	By Theorem~\ref{Thm:k-neighborsIV},
	we have the following $k$-neighbors of $\Gamma_{\IV}(n)$.
	
	\begin{tabular}{ll}
		For $k =1$:
		&
		$L_{k}^{\IV}(n)
		=
		\dfrac{2^{11}-2^5}{2^6-2^4}
		= 
		42.$\\
		For $k =2$:
		&
		$L_{k}^{\IV}(n)
		=
		\dfrac{(2^{11}-2^5)(2^9-2^5)}{(2^{6}-2^4)(2^6-2^2)}
		= 
		336.$\\
		For $k =3$:
		&
		$L_{k}^{\IV}(n)
		=
		\dfrac{(2^{11}-2^5)(2^9-2^5)(2^7-2^5)}{(2^{6}-2^4)(2^6-2^2)(2^6-1)}
		=
		512$
	\end{tabular}
	
	\noindent
	Then $1+42+336+512 = 891$ which is the total number of Type~$\IV$ codes.
\end{ex}

The following result is the Type~$\IV$ code analogue of Theorem~\ref{Thm:NumSlfDualkNeighbor}.

\begin{thm}\label{Thm:NumSlfDualkNeighborIV}
	Let~$n \equiv 0\pmod 2$.
	Then the number of Type~$\IV$ codes of length~$n$
	is $\sum_{k = 0}^{\frac{n}{2}} L_{k}^{\IV}(n)$.
\end{thm}

\begin{proof}
	Let $C_{1}$ and $C_{2}$ be any two 
	Type~$\IV$ codes of length~$n \equiv 0\pmod 4$.
	If $C_{1}$ and $C_{2}$ are connected by a path in~$\Gamma_{\IV}(n)$,
	then by Theorem~\ref{Thm:MaxDistanceTypeIV}, 
	the maximum path length will be~$\frac{n}{2}$.
	This completes the proof. 	
\end{proof}

\begin{ex}
	Let $n = 0\pmod 2$ be the length of the Type~$\IV$ codes,
	$|V_{\IV}(n)|$ the number of vertices in the graphs~$\Gamma_{\IV}(n)$.
	In Table~\ref{Tab:kNeighborTypeIV}, 
	we listed the $k$-neighbors of~$\Gamma_{\IV}(n)$ up to~$n = 10$. 
	Note that in each row $k$ goes from $0$ to $\frac{n}{2}$ and sum
	the sum of the $k$-neighbors in each row is~$|V_{\IV}(n)|$ as 
	in Theorem~\ref{Thm:NumSlfDualkNeighborIV}. 
	
	\begin{table}[h!]
		\begin{center}
			\caption{List of $k$-Neighbors in $\Gamma_{\IV}(n)$ up to~$n = 10$}
			\label{Tab:kNeighborTypeIV}
			\begin{tabular}{|c|c||c|c|c|c|c|c|}
				\hline
				\multirow{2}{*}{$n$} & \multirow{2}{*}{$|V_{\III}(n)|$} & \multicolumn{6}{c|}{$k$-neighbors}\\ \cline{3-8}
				&  & 0 & 1 & 2 & 3 & 4 & 5 \\
				\hline
				2 & 3 & 1 & 2 & 0 & 0 & 0 & 0\\
				4 & 27 & 1 & 10 & 16 & 0 & 0 & 0\\
				6 & 891 & 1 & 42 & 336 & 512 & 0 & 0\\
				8 & 114939 & 1 & 170 & 5712 & 43520 & 65536 & 0\\
				10 & 58963707 & 1 & 682 & 92752 & 2968064 & 22347776 & 33554432\\
%				12 & 120816635643 & 1 & 2730 & 1489488 & 192924160 & 6100942848 & 45801799680 & 68719476736 \\
				\hline\hline
			\end{tabular}
		\end{center}
	\end{table}
\end{ex}

\section{$k$-Neighbor graphs}\label{Sec:kNeighborGraphs}

\begin{df}
	Let $n \equiv 0\pmod 4$.
	Let $V_{\III}(n)$ be the set all Type~$\III$ codes 
	of length~$n$.
	Let $\Gamma_{\III}^{k}(n):=(V_{\III}(n),E_{\III}(n))$ be a graph, 
	where any two vertices in~$V_{\III}(n)$ are connected 
	by an edge in~$E_{\III}(n)$ if and only if they are $k$-neighbors. 
\end{df}

\begin{thm}\label{Thm:kNeighborGraph}
	Let $n \equiv 0\pmod 4$.
	Then the graph $\Gamma_{\III}^{k}(n)$
	satisfies the following properties.
	\begin{itemize}
		\item [(a)]
		The number of vertices is 
		$2\prod_{i=1}^{\frac{n}{2} - 1}(3^{i}+1)$.
		
		\item [(b)]
		The  graph is regular with degree
		$\dfrac{\prod_{i=0}^{k - 1} (3^{n-1-i}-3^{\frac{n}{2}-1})}{\prod_{j = 0}^{k-1}(3^{\frac{n}{2}}-3^{\frac{n}{2}-1-j})}$.
		
		\item [(c)]
		The number of edges is
		$\prod_{i=1}^{\frac{n}{2} - 1}(3^{i}+1)
		\dfrac{\prod_{i=0}^{k - 1} (3^{n-1-i}-3^{\frac{n}{2}-1})}{\prod_{j = 0}^{k-1}(3^{\frac{n}{2}}-3^{\frac{n}{2}-1-j})}$.
	\end{itemize} 
\end{thm}

\begin{proof}
	Theorems~\ref{Thm:GrVertexNumTypeIII} and~\ref{Thm:GrDegTypeIII}
	shows the statements (a) and (b), respectively.
	The proof of statement (c) is similar to the proof of Theorem~\ref{Thm:GrNeighborEdgeNum}.
\end{proof}

\begin{df}
	Let $n \equiv 0\pmod {12}$.
	Let $V_{\III}(n,\allone)$ be the set all Type~$\III$ codes 
	of length~$n$ containing~$\allone$.
	Let $\Gamma_{\III}^{k}(n,\allone):=(V_{\III}(n,\allone),E_{\III}(n,\allone))$ 
	be a graph, where any two vertices in~$V_{\III}(n,\allone)$ are connected 
	by an edge in~$E_{\III}(n,\allone)$ if and only if they are $k$-neighbors. 
\end{df}

\begin{thm}\label{Thm:kNeighborGraphAllOne}
	Let $n \equiv 0\pmod {12}$.
	Then the graph $\Gamma_{\III}^{k}(n,\allone)$
	satisfies the following properties.
	\begin{itemize}
		\item [(a)]
		The number of vertices is 
		$2\prod_{i=1}^{\frac{n}{2} - 2}(3^{i}+1)$.
		
		\item [(b)]
		The  graph is regular with degree
		$\dfrac{\prod_{i=0}^{k - 1} (3^{n-2-i}-3^{\frac{n}{2}-1})}{\prod_{j = 0}^{k-1}(3^{\frac{n}{2}}-3^{\frac{n}{2}-1-j})}$.
		
		\item [(c)]
		The number of edges is
		$\prod_{i=1}^{\frac{n}{2} - 2}(3^{i}+1)
		\dfrac{\prod_{i=0}^{k - 1} (3^{n-2-i}-3^{\frac{n}{2}-1})}{\prod_{j = 0}^{k-1}(3^{\frac{n}{2}}-3^{\frac{n}{2}-1-j})}$.
	\end{itemize} 
\end{thm}

\begin{proof}
	Theorems~\ref{Thm:GrVertexNumTypeIIIAllOne} and~\ref{Thm:GrDegTypeIIIAllOne}
	shows the statements (a) and (b), respectively.
	The proof of statement (c) is similar to the proof of Theorem~\ref{Thm:GrNeighborEdgeNumAllOne}.
\end{proof}

\begin{df}
	Let $n \equiv 0\pmod 2$.
	Let $V_{\IV}(n)$ be the set all Type~$\IV$ codes 
	of length~$n$.
	Let $\Gamma_{\IV}^{k}(n):=(V_{\IV}(n),E_{\IV}(n))$ be a graph, 
	where any two vertices in~$V_{\IV}(n)$ are connected 
	by an edge in~$E_{\IV}(n)$ if and only if they are neighbors. 
\end{df}

\begin{thm}\label{Thm:kNeighborGraphIV}
	Let $n \equiv 0\pmod 2$.
	Then the graph $\Gamma_{\IV}^{k}(n)$
	satisfies the following properties.
	\begin{itemize}
		\item [(a)]
		The number of vertices is 
		$\prod_{i=0}^{\frac{n}{2} - 1}(2^{2i+1}+1)$.
		
		\item [(b)]
		The  graph is regular with degree
		$\dfrac{\prod_{i=0}^{k - 1} (2^{2n-1-2i}-2^{n-1})}{\prod_{j = 0}^{k-1}(2^{n}-2^{n-2-2j})}$.
		
		\item [(c)]
		The number of edges is
		$\frac{1}{2}\prod_{i=0}^{\frac{n}{2} - 1}(2^{2i+1}+1)
		\dfrac{\prod_{i=0}^{k - 1} (2^{2n-1-2i}-2^{n-1})}{\prod_{j = 0}^{k-1}(2^{n}-2^{n-2-2j})}$.
	\end{itemize} 
\end{thm}

\begin{proof}
	Theorems~\ref{Thm:GrVertexNumTypeIV} and~\ref{Thm:GrDegTypeIV}
	shows the statements (a) and (b), respectively.
	The proof of statement (c) is similar to the proof of Theorem~\ref{Thm:GrNeighborEdgeNumIV}.
\end{proof}

\section{An application of neighbors in invariant theory}\label{Sec:Application}

In this section, 
we investigate the invariant ring of weight enumerators 
for Type~$\II$ codes in genus $g$, 
with particular emphasis on identifying the generators of the ring 
using the concept of neighbors.
A binary self-dual code~$C$ is
called \emph{Type~$\II$} 
if the weight of each codeword of~$C$ is a multiple of~$4$.
It is known that 
a Type~$\II$ code of length~$n$ exists if and only 
if~$n \equiv 0 \pmod 8$.
There are~$9$ Type~$\II$ codes of
length~$24$ up to equivalence, 
denoted by $C_{i}$ for $i = 1, 2, \ldots , 9$.
We present these codes
in Table~3. 
For detail discussion about $C_{i}$'s, 
we refer the reader 
to~\cite{MR1168153, MR2967223, PS1975}.

By $d_n$ and $d_{n}^{+}$, we denote the code
with following generator matrices
for $n \equiv 0 \pmod{8}$:

\[
	d_{n}:
	\begin{pmatrix}
		11110000 & \cdots & 0000 \\
		11001100 & \cdots & 0000 \\
		11000011 & \cdots & 0000 \\
		\vdots & \ddots & \vdots \\
		11000000 & \cdots & 0011 
	\end{pmatrix},
\quad\quad
	d_{n}^{+}:
	\begin{pmatrix}
		11110000 & \cdots & 0000 \\
		11001100 & \cdots & 0000 \\
		\vdots & \ddots & \vdots \\
		11000000 & \cdots & 0011 \\
		10101010 & \cdots & 1010 
	\end{pmatrix}.
	\]
In particular, $d_{8}^{+}$ is denoted by~$e_{8}$.
Additionally, 
$g_{24}$ denotes the binary Golay code of length~$24$, 
which is the unique Type~$\II$ code of this length 
that does not include any elements of weight 4, see~\cite{PS1975}.

\begin{table}[ht]
	\caption{Classification of Type~$\II$ codes of length~$24$}
	\centering
	\begin{tabular}{l|ccccccccc}
		Code & $C_1$ & $C_2$ & $C_3$ & $C_4$ & $C_5$ & $C_6$ & $C_7$ & $C_8$ & $C_9$ \\
		\hline\\ [-1em]
		Components & $d_{12}^{2}$ & $d_{10}e_7^2$ & $d_8^3$ & $d_6^4$ & $d_{24}$ & $d_4^6$ & $g_{24}$ & $d_{16}e_8$ & $e_8^3$ 
	\end{tabular}
\end{table}

Next we recall the definitions and known facts from invariant theory.
Here we prefer to denote an element of $\FF_2^g$ 
by a column vector.
Let~$C$ be a Type~$\II$ code of length~$n$.
Then the \emph{weight enumerator}~$C$ in genus~$g$ is:
\[
W_{{C}}^{(g)}(x_a:\ a\in\FF_2^g)
=
\sum_{u,v\in {C}}
\prod_{a\in \FF_2^g}
x_a^{n_a\begin{pmatrix} u_1\\ \vdots \\ u_g\end{pmatrix}},
\]
where $n_a\begin{pmatrix} u_1\\ \vdots \\ u_g\end{pmatrix}$
is the number of $i$ such that 
$\begin{pmatrix} u_{1i}\\ \vdots \\ u_{gi}\end{pmatrix} = a$.
Now
let us use the following notations for various rings in our discussion:
\begin{center}
	\begin{tabular}{lcl}
		$\mathfrak{B}^{(g)}$
		& : & the ring of $W_{C}^{(g)}$, where $C$ is Type~$\II$,\\
		$\mathfrak{D}^{(g)}$
		& : & the ring of $W_{d_{n}^{+}}^{(g)}$, where $n \equiv 0\pmod 8$,\\
		$\mathfrak{A}^{(g)}$
		& : & the ring of $W_{C}^{(g)}$, where $C$ is $d_{n}^{+}$ and its neighbors.
	\end{tabular}
\end{center}
Clearly, $\mathfrak{D}^{(g)} \subseteq \mathfrak{A}^{(g)} \subseteq \mathfrak{B}^{(g)}$.
Since $\mathfrak{B}^{(g)}$ and $\mathfrak{D}^{(g)}$ are finitely generated over~$\CC$, see~\cite{MR1219862, FuOu}, 
it follows that $\mathfrak{A}^{(g)}$ is as well.
%Fuji and Oura
It is proved in~\cite[Proposition 2]{FuOu} 
%considered the ring generated by the weight enumerators of $d_n^+$
%for arbitrary genus. 
that 
$\mathfrak{D}^{(1)} = \mathfrak{B}^{(1)}$,
however $\mathfrak{D}^{(2)}$ is strictly smaller than $\mathfrak{B}^{(2)}$. 
In this note, we would like to discuss on 
the ring~$\mathfrak{A}^{(g)}$ of weight enumerators of Type~$\II$ code $d_n^+$ 
and its neighbors. Table~$4$ gives neighbors of code~$d_{n}^{+}$
for $n = 8, 16, 24$.

\begin{table}[ht]
	\caption{$d_{n}^{+}$ and its neighbors up to length~$24$}
	\centering
	\begin{tabular}{c | l}
		Code & Neighbors (up to equivalence)\\
		\hline\\ [-1em]
		$e_{8}$ & $e_{8}$\\ 
		$d_{16}^{+}$ & $d_{16}^{+}$\\ 
		$d_{24}^{+}$ & $C_{1}$, $C_{5}$, $C_{8}$
	\end{tabular} 
	\label{Tab:Neighbors}
\end{table}

Now let us define following matrices 
in \( \mathrm{GL}(2^g, \mathbb{C}) \): 
\begin{align*}
	T_g &= \left(\frac{1+i}{2}\right)^g \left({(-1)}^{(a,b)}\right)_{a,b \in \mathbb{F}_2^g}, \\
	D_S &= \text{diag}(i^{S[a]} \text{ for } a \in \mathbb{F}_2^g),
\end{align*}
where, $S[a]:= {^ta}Sa$ for any symmetric \( g \times g \) matrix \( S \).
%Let \( H_g \) be a subgroup of $\mathrm{GL}(2^g,\CC)$
%generated by $T_g$ \text{ and } $D_S$.
Let 
\[ 
	G_g := \langle T_{g}, D_{S}, \zeta_8\rangle
\]
be a subgroup of $\mathrm{GL}(2^g,\CC)$ generated by 
$T_{g}$, $D_{S}$ and \( \zeta_8 \), 
where~$S$ runs over all symmetric matrices of order~$g$
and \( \zeta_8 = e^{2\pi i / 8} \) is the primitive $8$th root of unity.
The order of the group~$G_{g}$ for $g = 1, 2, 3$ are shown in 
Table~$5$.
The group $G_g$ acts naturally on the polynomial ring $\mathbb{C}[x_a] := \mathbb{C}[x_a : a \in \FF_2^g]$. 
We denote $\mathbb{C}[x_a]^{G_g}$ the invariant ring under the action of~$G_{g}$.

\begin{table}[ht]
	\caption{Order of ${G}_{g}$}
	\centering
	\begin{tabular}{c | ccc}
		$g$ & 1  & 2 & 3\\
		\hline\\ [-1em]
		$|{G}_{g}|$ & 192 & 92160 & 743178240\\
	\end{tabular} 
	\label{Tab:OrderGgell}
\end{table}

We recall~\cite{MR1219862, MR0424391, NRS} for 
the dimension formulae of the invariant ring $\mathbb{C}[x_a]^{G_g}$ 
for $g = 1, 2, 3$ as follows:
\begin{align*}
	&g =1:
	\frac{1}{(1 - t^8)(1 - t^{24})} = 1 + t^8 + t^{16} + 2t^{24} + \cdots,\\
	&g =2:
	\frac{1 + t^{32}}{(1 - t^8)(1 - t^{24})^2(1 - t^{40})} = 1 + t^8 + t^{16} + 3t^{24} + \cdots, \\
	&g = 3:
	\frac{\theta(t^8) + t^{352}\theta(t^{-8})}{(1 - t^8)(1 - t^{16})(1 - t^{24})^2(1-t^{40})(1 -t^{56}) (1-t^{72})(1-t^{120})} \\
	& \quad \quad \quad = 1 + t^8 + 2t^{16} + 5t^{24} + \cdots.
\end{align*}
 where
 \begin{align*}
 	\theta(t) := &\; 1 + t^3 + 3t^4 + 3t^5 + 6t^6 + 8t^7 + 12t^8 + 18t^9 + 25t^{10} \\
 	& + 29t^{11} + 40t^{12} + 50t^{13} + 58t^{14} + 69t^{15} + 80t^{16} \\
 	& + 85t^{17} + 96t^{18} + 104t^{19} + 107t^{20} + 109t^{21} + 56t^{22}.
 \end{align*}

%\begin{thm}[\cite{MR1368288, MR0424391}]
%	The invariant ring $\mathbb{C}[x_a]^{G_g}$ is generated by the weight enumerators of Type~$\II$ codes in genus $g$.
%\end{thm}

It is known that 
the invariant ring~$\mathbb{C}[x_a]^{G_g}$
is generated by the weight enumerators of Type~$\II$ codes 
in genus~$g$, see~\cite{MR1219862, MR0424391, MR1368288}. 
In particular, a basis of the vector space generated by the
weight enumerators of Type~$\II$ code of length~$24$ 
in $g = 1, 2, 3$ is given below:
\begin{align*}
	g = 1
	&:
	W_{C_{9}}^{(1)}, W_{C_{7}}^{(1)}\\
	g = 2
	&:
	W_{C_{9}}^{(2)}, W_{C_{7}}^{(2)}, W_{C_{5}}^{(2)}\\
	g = 3
	&:
	W_{C_{9}}^{(3)}, W_{C_{7}}^{(3)}, W_{C_{5}}^{(3)}, W_{C_{8}}^{(3)}, W_{C_{1}}^{(3)}.
\end{align*}

It is can be seen from Table~$3$ that $C_{5}$ 
is $d_{24}^+$ itself.
Moreover, from Table~$4$ we have 
$C_1$, $C_{5}$ and  $C_8$ being the neighbors of $d_{24}^+$. 
Since~$W_{C_{9}}^{(1)}$ and $W_{C_{5}}^{(1)}$ are algebraically independent,
therefore 
we have $\mathfrak{A}^{(1)} = \CC[W_{e_{8}}^{(1)},W_{d_{24}^{+}}^{(1)}] = \CC[x_{a}]^{G_{1}}$. 

The above discussions and \cite[Proposition 2]{FuOu}, conclude
the following result.

\begin{thm}
	$\mathfrak{D}^{(1)}=\mathfrak{A}^{(1)}=\mathfrak{B}^{(1)}$
	up to the space of degree~$24$.
\end{thm}

\begin{thm}
	$\mathfrak{D}^{(2)}\subsetneq\mathfrak{A}^{(2)}=\mathfrak{B}^{(2)}$
	up to the space of degree~$24$.    
\end{thm}

\begin{proof}
	The code $d_{24}^+$ has a neighbor $C_1$. 
	From the dimension formula, 
	we know that the space of degree~$24$ in~$\mathfrak{B}^{(2)}$
	has dimension~$3$. 
	Since $W_{g_{24}}^{(2)} \notin \mathfrak{A}^{(2)}$,
	it is enough to show that
	$W_{C_{1}}^{(2)}$ belongs to the basis 
	of the space of degree~$24$ in~$\mathfrak{A}^{(2)}$.
	Therefore, 
%	to show $\mathfrak{A}^{(2)}=\mathfrak{B}^{(2)}$,
	we consider the genus~$2$ weight enumerators 
	of $e_8^3$, $d_{24}^+$ and~$C_1$.
	We select the following monomials of these weight enumerators: 
	\[
	\begin{tabular}{llll}
		$\alpha x_{\begin{smallmatrix}0\\ 0\end{smallmatrix}}^{24}$,
		&
		$\beta x_{\begin{smallmatrix}0\\ 0\end{smallmatrix}}^{20}
		x_{\begin{smallmatrix}0\\ 1\end{smallmatrix}}^4$,
		&
		$\gamma x_{\begin{smallmatrix}0\\ 0\end{smallmatrix}}^{16}
		x_{\begin{smallmatrix}0\\ 1\end{smallmatrix}}^8$,
	\end{tabular}
	\]
	where $\alpha$, $\beta$ and $\gamma$ 
	represent the coefficients of the  monomials.
	Now  we construct the following $3 \times 3$ matrix $L$
	consisting of the coefficients of 
	the aforementioned~$3$ monomials 
	from the selected weight enumerators:
	
	\[
	\begin{array}{c|ccc}
		\text{Code} & \alpha & \beta & \gamma\\
		\hline\\ [-1em]
		e_{8}^{3} & 1 & 42 & 591 \\
		d_{24}^{+} & 1 & 66 & 495 \\
		C_{1} & 1 & 30 & 639
	\end{array}
	\]
	Immediately, $\mathrm{Rank}(L) = 3$.
	It is known (see~\cite{MR1219862}) that 
	the ring 
	%$\mathbb{C}[X_a]^{G_2}$
	$\mathfrak{B}^{(2)}$ 
	has the following structure:
	\[
		\CC[W_{e_{8}}^{(2)},W_{d_{24}^{+}}^{(2)},W_{g_{24}}^{(2)},W_{d_{40}^{+}}^{(2)}]
			\oplus
		\CC[W_{e_{8}}^{(2)},W_{d_{24}^{+}}^{(2)},W_{g_{24}}^{(2)},W_{d_{40}^{+}}^{(2)}]
		W_{d_{32}^{+}}^{(2)}.
	\]
	Since $\mathrm{Rank}(L) = 3$,
	the weight enumerators of 
	$e_8^3$, $d_{24}^+$ and $C_1$ are algebraically independent. 
	Thus the space of degree~$24$ in $\mathfrak{A}^{(2)}$ 
	is of dimension~$3$ and is same as $\mathfrak{B}^{(2)}$.
	Moreover, $\mathfrak{D}^{(2)} \subsetneq \mathfrak{B}^{(2)}$, 
	see~\cite{FuOu}.
	This completes the proof.
\end{proof}

Our computation shows that 
the dimensions and a basis of the spaces 
of degrees $8$, $16$ and $24$ in $\mathfrak{D}^{(3)}$
are as follows:

\begin{table}[ht]
	\caption{Dimension and basis of $\mathfrak{D}^{(3)}$ up to length~$24$}
	\centering
	\begin{tabular}{c | c | l}
		Length & Dimension & Basis (up to equivalence)\\
		\hline\\ [-1em]
		$8$ & 1 & $e_{8}$\\ 
		$16$ & 2 & $e_{8}^{2}$, $d_{16}^{+}$\\ 
		$24$ & 3 & $C_{9}$, $C_{5}$, $C_{8}$
	\end{tabular} 
	\label{Tab:BasisD3}
\end{table}

\begin{thm}
	$\mathfrak{D}^{(3)}\subsetneq\mathfrak{A}^{(3)}\subsetneq\mathfrak{B}^{(3)}$
	up to the space of degree~$24$.    
\end{thm}

\begin{proof}
	 From Table~$4$, we have $C_{1}$, $C_{5}$ and $C_{8}$
	 are the neighbors of the code~$d_{24}^{+}$. 
	 Clearly, $C_5$ is $d_{24}^+$ itself. 
	 Since $W_{C_{8}}^{(3)}$ belongs to the basis 
	 of the space of degree~$24$ in~$\mathfrak{B}^{(3)}$,
	 it follows that is true for $\mathfrak{A}^{(3)}$ as well.
	 Since $W_{g_{24}}^{(3)} \notin \mathfrak{A}^{(3)}$,
	 it is enough to show that
	 $W_{C_{1}}^{(3)}$ belongs to the basis 
	 of the space of degree~$24$ in~$\mathfrak{A}^{(3)}$.
	 Therefore, we consider the genus~$3$ weight enumerators 
	 of $e_8^3$, $d_{24}^+$, $C_{8}$ and~$C_1$.
	 We select the following monomials of these weight enumerators: 
	\[
		\begin{tabular}{llll}
			$\alpha x_{\begin{smallmatrix}0\\ 0 \\ 0\end{smallmatrix}}^{20}
			x_{\begin{smallmatrix}0\\ 1 \\ 1\end{smallmatrix}}^4$,
			&
			$\beta x_{\begin{smallmatrix}0\\ 1 \\ 0\end{smallmatrix}}^{16}
			x_{\begin{smallmatrix}1\\ 0 \\ 1\end{smallmatrix}}^8$,
			&
			$\gamma x_{\begin{smallmatrix}0\\ 0 \\ 1\end{smallmatrix}}^8
			x_{\begin{smallmatrix}1\\ 0 \\ 0\end{smallmatrix}}^4
			x_{\begin{smallmatrix}1\\ 1 \\ 0\end{smallmatrix}}^{12}$,
			&
			$\delta x_{\begin{smallmatrix}0\\ 0 \\ 0\end{smallmatrix}}^4
			x_{\begin{smallmatrix}0\\ 1 \\ 0\end{smallmatrix}}^2
			x_{\begin{smallmatrix}0\\ 0 \\ 1\end{smallmatrix}}^6
			x_{\begin{smallmatrix}1\\ 0 \\ 0\end{smallmatrix}}^6
			x_{\begin{smallmatrix}1\\ 1 \\ 1\end{smallmatrix}}^6$,
		\end{tabular}	
	\]
	with $\alpha, \beta, \gamma, \delta$ being the coefficients of the monomials. 
	Now we construct the following $4\times 4$ matrix $M$
	consisting of the coefficients of 
	the aforementioned~$4$ monomials 
	from the selected weight enumerators:
	
	\[
	\begin{array}{l|cccc}
		\text{Code} & \alpha & \beta & \gamma & \delta \\
		\hline
		e_8^3       & 42 & 591 & 9491 & 592704 \\
		d_{24}^+ & 66 & 495 & 13860 & 110800 \\
		C_{8}  & 42 & 591 & 9492 & 762048 \\
		C_1 & 30 & 639 & 7020 & 659520 \\
	\end{array}
	\]
	 It is immediate that $\mathrm{Rank}(M) = 4$.
	This means 
	$W_{e_{8}^{3}}, W_{d_{24}^+}, W_{C_{8}}, W_{C_1}$ 
	are algebraically independent and form a dimension 4 vector space. 
	Thus the space of degree~$24$ in $\mathfrak{A}^{(3)}$ 
	is of dimension~$4$ and hence it is strictly smaller than 
	$\mathfrak{B}^{(3)}$. 
	Moreover, 
	the space of degree~$24$ in $\mathfrak{D}^{(3)}$
	is of dimension~$3$ (see Table~$6$),
	This completes the proof.
\end{proof}

\section*{Declaration of competing interest}

The authors declare that 
they have no known competing financial interests or personal relationships 
that could have appeared to influence the work reported in this paper.

\section*{Acknowledgements}
\noindent
%The authors thank Thomas Britz
%Koji Chinen and Iwan Duursma 
%for helpful discussions. 
%and contributions to this research.
%The authors would also like to thank the anonymous reviewers 
%for their beneficial comments on 
%an earlier version of the manuscript. 
%The second author is supported by JSPS KAKENHI (22K03277). 
This work was supported by JSPS KAKENHI Grant Numbers 22K03277,
24K06827 
and 
SUST Research Centre under 
Project ID
PS/2023/1/22.

\section*{Data availability statement}
The data that support the findings of this study are available from
the corresponding author.

\end{document}